\documentclass[11pt]{article}
\usepackage{amsmath,amssymb,bm,amsthm}
\def\N{\mathbb{N}}
\def\Z{\mathbb{Z}}
\def\C{\mathbb{C}}
\def\R{\mathbb{R}}
\def\L{\mathcal{L}}

\def\G{\mathcal{G}}
\def\CB{\mathcal{CB}}
\def\a{\mathbf{a}}
\def\b{\mathbf{b}}
\def\c{\mathbf{c}}
\def\d{\mathbf{d}}
\def\O{\mathcal{O}}
\def\m{\mathbf{m}}
\DeclareMathOperator*{\res}{res}
\newtheorem{theorem}{\hspace*{\parindent}Theorem}
\newtheorem{propos}{\hspace*{\parindent}Proposition}
\newtheorem{lemma}{\hspace*{\parindent}Lemma}
\newtheorem{corollary}{\hspace*{\parindent}Corollary}
\title{A class of Meijer's $G$ functions and further representations of the generalized hypergeometric functions}
\author{D.B.\:Karp$^{\rm a,b}$\footnote{Corresponding author. E-mail: D.\:Karp -- \emph{dimkrp@gmail.com}, J.L.\:L\'{o}pez --  \emph{jl.lopez@unavarra.es}}~~and J.L.\:L\'{o}pez$^{\rm c}$
\\[10pt]\small{\textit{$\phantom{1}^a$Far Eastern Federal University, Vladivostok, Russia}}
\\
\small{\textit{$\phantom{1}^b$Institute of Applied Mathematics FEBRAS, Vladivostok, Russia}}
\\
\small{\textit{$\phantom{1}^c$Dpto. de Ingenier\'{\i}a Matem\'{a}tica e Inform\'{a}tica,}}
\\
\small{\textit{Universidad P\'{u}blica de Navarra and INAMAT, Navarra, Spain}}}
\date{}
\begin{document}

\maketitle

\begin{center}
\parbox{12cm}{
\small\textbf{Abstract.}  In this paper we investigate the Meijer's $G$ function $G^{p,1}_{p+1,p+1}$ which for certain parameter values represents the Riemann-Liouville fractional
integral of Meijer-N{\o}rlund function $G^{p,0}_{p,p}$. Our results for $G^{p,1}_{p+1,p+1}$ include: a regularization formula for overlapping poles,
a connection formula with the Meijer-N{\o}rlund function, asymptotic formulas around the origin and unity, formulas for the moments, a hypergeometric transform and a sign
stabilization theorem for growing parameters. We further employ the properties of $G^{p,1}_{p+1,p+1}$ to calculate the Hadamard finite part of an integral
containing the Meijer-N{\o}rlund function that is singular at unity.  In the ultimate section, we define an alternative regularization for such integral better suited
for representing the Bessel type generalized hypergeometric function ${}_{p-1}F_{p}$. A particular case of this regularization is then used to identify some new facts
about the positivity and reality of the zeros of this function.}
\end{center}

\bigskip

Keywords: \emph{Meijer's $G$ function, generalized hypergeometric function, integral representation, Hadamard finite part}

\bigskip

MSC2010: 33C60, 33C20, 26A33, 46F99

\bigskip

\section{Introduction}
Throughout the paper we will use the standard notation ${_{p}F_q}$  for the generalized hypergeometric function (see \cite[Section~2.1]{AAR}, \cite[Section~5.1]{LukeBook},
\cite[Sections 16.2-16.12]{NIST} or \cite[Chapter~12]{BealsWong})  and $G^{m,n}_{p,q}$ for the Meijer's $G$ function (see \cite[section~5.2]{LukeBook}, \cite[16.17]{NIST}, \cite[8.2]{PBM3}
or \cite[Chapter~12]{BealsWong}).  The role that the Meijer's $G$ function plays for integral representations of the generalized hypergeometric functions was probably first recognized by
V.\:Kiryakova in \cite[Chapter~4]{KiryakovaBook} and \cite{Kiryakova97}, through the use of successive fractional integrations. In a series of papers
\cite{KaKaAnMAth2017,KarpJMS,KPJMAA,KPITSF2017,KSJAT} the first author jointly with Kalmykov, Prilepkina and Sitnik extended Kiryakova's results and applied them to discover numerous new facts about the generalized hypergeometric functions.
This work has been continued in our recent article \cite{KLMeijer1}, where the properties of the Meijer-N{\o}rlund function $G^{p,0}_{p,p}$ were employed to
investigate the connections of the generalized hypergeometric functions with topics like  inverse factorial series, radial positive definite functions, Luke's inequalities and
zero-free regions.  We further extended the known representations to arbitrary values of parameters by studying a regularization of integrals of the form
\begin{equation}\label{eq:Gphiintegral}
\int_0^1\!\!G_0(t)\varphi(t)dt,~~\text{where}~~G_0(t)=G^{p,0}_{p,p}\!\left(t\left\vert\begin{matrix}\b-1\\\a-1\end{matrix}\right.\right)
\end{equation}
and $\varphi$ is a ''nice'' function (see details below). For certain values of the parameter vectors $\a=(a_1,\ldots,a_p)$ and $\b=(b_1,\ldots,b_p)$ the
function $G_0(t)$ has non-integrable power singularities at $t=0$ and/or $t=1$. The regularization from \cite[Section~5.2]{KLMeijer1} uses the Taylor coefficients
of the function $\varphi(t)$  at the point $t=0$ to overcome this divergence.

In this paper we continue our study of the representations of the generalized hypergeometric functions via integrals of Meijer's $G$ function begun in \cite{KLMeijer1}.
The main emphasis in this work will be made on the regularization of (\ref{eq:Gphiintegral}) based on the expansion of $\varphi(t)$ in the neighborhood of $t=1$.
This requires a detailed investigation of the properties of the  function $G^{p,1}_{p+1,p+1}$ which are less known than the properties of the Meijer-N{\o}rlund function $G^{p,0}_{p,p}$.
We conduct such investigation in Section~2.  In particular, we derive an identity relating  $G^{p,1}_{p+1,p+1}$ with $G^{p+1,0}_{p+1,p+1}$ (under a certain restriction on parameters),
a regularization formula for $G^{p,1}_{p+1,p+1}$ when the poles of the integrand of different types superimpose, an expression for the moments of the function $G^{p,1}_{p+1,p+1}$
and a formula for its hypergeometric transform which incorporates generalized Stieltjes, Laplace and Hankel transforms. Furthermore, we prove a proposition on sign stabilization
for $G^{p,1}_{p+1,p+1}$ when all but one parameters grow infinitely.  In Section~3 we utilize the new properties from Section~2 to define and study a regularization of the
integral (\ref{eq:Gphiintegral}) that uses Taylor coefficients of $\varphi(t)$ at the point $t=1$. Applying this regularization method to generalized Stieltjes, Laplace and
cosine Fourier kernels we obtain new integral representations of the generalized hypergeometric functions.  Finally, in the ultimate Section~4 we define an
alternative regularization of (\ref{eq:Gphiintegral}) tailored to better serve the generalized hypergeometric function of the Bessel type.  This approach leads to
new information about positivity and real zeros of this function also presented in the same section.

\section{New properties of the Meijer's function $G^{p,1}_{p+1,p+1}$}

Let us fix some notation and terminology first.  The standard symbols $\N$, $\Z$, $\R$ and $\C$ will be used to denote the natural, integer, real and complex numbers, respectively;
$\N_0=\N\cup\{0\}$.  In what follows we will use the shorthand notation  for the products and sums:
\begin{equation*}
\Gamma(\a)=\Gamma(a_1)\Gamma(a_2)\cdots\Gamma(a_p),~~(\a)_n=(a_1)_n(a_2)_n\cdots(a_p)_n,~~\a+\mu=(a_1+\mu,a_2+\mu,\dots,a_p+\mu);
\end{equation*}
inequalities like $\Re(\a)>0$ and properties like $-\a\notin\N_0$ will be understood element-wise (i.e. $-\a\notin\N_0$ means that
no element of $\a$ is non-positive integer). The key role in our construction will be played by the function
\begin{equation}\label{eq:Gn-defined}
\widehat{G}_n(t)\!:=\!G^{p,1}_{p+1,p+1}\!\!\left(\!t\left|\begin{array}{l}\!\!n,\b+n-1\!\!\\\!\!\a+n-1,0\!\!\end{array}\right.\!\right)
\!=\!\frac{1}{2\pi{i}}
\int\limits_{\mathcal{L}}\!\!\frac{\Gamma(\a+n-1+s)\Gamma(1-n\!-\!s)}{\Gamma(\b+n-1+s)\Gamma(1-s)}t^{-s}ds,
\end{equation}
where $\a=(a_1,\ldots,a_p)$, $\b=(b_1,\ldots,b_p)$ are (generally complex) parameter vectors. For $|t|<1$ the contour $\L$ is a left loop that separates the poles of the integrand of the
form $a_{jl}=1-a_j-n-l$, $l\in\N_0$, leaving them on the left from the poles of the form $1-n+k$, $k\in\N_0$, leaving them on the right. By definition the two types of poles must not
superimpose, which translates into the condition $-a_j\notin\N_0$. If they do, the definition can still be repaired by the regularization given in Proposition~\ref{lm:Gn-extendeddef}
below.  Further details regarding the choice of the contour and convergence of the integral can be found, for instance, in  \cite[section~1.1]{KilSaig} or in \cite[section~2]{KLMeijer1}.
Definition (\ref{eq:Gn-defined}) certainly works for any complex $n$, but for our purposes we confine ourselves to $n\in\N_0$. Note, that due to the shifting
property (see \cite[8.2.2.15]{PBM3} or \cite[16.19.2]{NIST})
$$
t^{\alpha}G^{m,n}_{p,q}\!\left(\!t~\vline\begin{array}{l}\b\\\a\end{array}\!\!\right)=G^{m,n}_{p,q}\!\left(\!t~\vline\begin{array}{l}\b+\alpha\\\a+\alpha\end{array}\!\!\right),
$$
any function $G^{p,1}_{p+1,p+1}(t)$ can be written as (\ref{eq:Gn-defined}) times some power of $t$ if $n$ is allowed to be complex. The restriction $n\in\N_0$ means that
the top left parameter must be greater than the bottom right parameter by a nonnegative integer.  We also recall the next definition from \cite[(34)]{KLMeijer1}:
\begin{equation}\label{eq:tildeGn-defiend}
\widetilde{G}_n(t)=G^{p+1,0}_{p+1,p+1}\left(t\left|\begin{array}{l}\!\!\b-1+n,n\!\!\\\!\!\a-1+n,0\!\!\end{array}\right.\right).
\end{equation}
It is straightforward that $\widehat{G}_0(t)=\widetilde{G}_0(t)=G_0(t)$, where $G_0(t)$ is defined in (\ref{eq:Gphiintegral}).  Define
\begin{equation}\label{eq:a-psi-defined}
a:=\min(\Re{a_1},\Re{a_2},\ldots,\Re{a_p})~~\text{and}~~\psi:=\sum\nolimits_{k=1}^{p}(b_k-a_k).
\end{equation}
If $\Re(\a)>0$ the function $\widehat{G}_n(t)$ can be computed as the $n-$th primitive of $G_0(x)$ that satisfies $\widehat{G}_n^{(k)}(0)=0$ for $k=1,2,\ldots,n$
(see details in Proposition~\ref{lm:Gnfracintegral} below):
\begin{equation}\label{eq:hatGnG0}
\widehat{G}_n(t)=\frac{1}{(n-1)!}\int_0^tG_0(x)(t-x)^{n-1}dx.
\end{equation}
As mentioned above the function $\widehat{G}_n(t)$ is not defined if any component of $\a$ is a non positive integer, since basic separation condition for the
contour $\L$ in (\ref{eq:Gn-defined}) is violated.  However, as we will see, the function $\widehat{G}_n(t)/\Gamma(\a)$ is entire in $\a$. \
If $\widehat{G}_n(t)/\Gamma(\a)$ is viewed as a function of one parameter (say $a_1$) and all elements of $\a_{[1]}=(a_2,\ldots,a_p)$ are different modulo integers, then
this claim follows from the representation \cite[(16.17.2)]{NIST}
\begin{equation}\label{eq:Gnhypergeomexp}
\frac{\widehat{G}_n(t)}{\Gamma(\a)}=\sum\limits_{k=1}^{p}\frac{\Gamma(\a_{[k]}-a_k)t^{a_k+n-1}}{\Gamma(\b-a_k)\Gamma(a_k+n)\Gamma(\a_{[k]})}
{}_{p+1}F_p\left.\left(\begin{matrix}a_k,1+a_k-\b\\a_k+n,1+a_k-\a_{[k]}\end{matrix}\right\vert t\right),
\end{equation}
where here and in the sequel $\a_{[k]}:=(a_1,\ldots,a_{k-1},a_{k+1},\ldots,a_p)$. However, in case of multiple poles (i.e. when some of the differences $a_i-a_k\in\Z$)
the situation becomes more delicate. In order to treat the general case we will need the following statement which shows that $\widetilde{G}_n(x)$ defined in
(\ref{eq:tildeGn-defiend}) and $(-1)^n\widehat{G}_n(x)$ differ by a polynomial.
\begin{propos}
The following identity holds true
\begin{equation}\label{eq:GGF}
\widetilde{G}_n(x)-(-1)^n\widehat{G}_n(x)=\frac{(-x)^{n-1}\Gamma(\a)}{(n-1)!\Gamma(\b)}{}_{p+1}F_p\left.\!\left(\begin{matrix}-n+1,\a\\\b\end{matrix}\right\vert \frac{1}{x}\right),
\end{equation}
where $\widetilde{G}_n$ and $\widehat{G}_n$ are defined in \emph{(\ref{eq:tildeGn-defiend})} and \emph{(\ref{eq:Gn-defined})}, respectively.
\end{propos}
\begin{proof} Assuming that $\Re(\a)$ and $\Re(\psi)$ are positive and substituting the definitions of  $\widetilde{G}_n$ and $\widehat{G}_n$ into the left hand side of (\ref{eq:GGF}),
in view of representation \cite[(41)]{KLMeijer1}, we get
\begin{multline*}
G^{p+1,0}_{p+1,p+1}\!\!\left(x\left|\begin{array}{l}\!\!n,\b+n-1\!\!\\\!\!\a+n-1,0\!\!\end{array}\right.\right)-
(-1)^nG^{p,1}_{p+1,p+1}\!\!\left(x\left|\begin{array}{l}\!\!n,\b+n-1\!\!\\\!\!\a+n-1,0\!\!\end{array}\right.\right)
\\
=\frac{1}{(n-1)!}\left\{\int_0^x(t-x)^{n-1}G^{p,0}_{p,p}\left(t\left\vert\begin{matrix}\b-1\\\a-1\end{matrix}\right.\right)dt
+\int_x^1(t-x)^{n-1}G^{p,0}_{p,p}\left(t\left\vert\begin{matrix}\b-1\\\a-1\end{matrix}\right.\right)dt\right\}
\\
=\frac{1}{(n-1)!}\left\{\int_0^x\left[\sum_{j=0}^{n-1}\binom{n-1}{j}t^j(-x)^{n-1-j}\right]G^{p,0}_{p,p}\left(t\left\vert\begin{matrix}\b-1\\\a-1\end{matrix}\right.\right)dt\right.
\\
\left.+\int_x^1\left[\sum_{j=0}^{n-1}\binom{n-1}{j}t^j(-x)^{n-1-j}\right]G^{p,0}_{p,p}\left(t\left\vert\begin{matrix}\b-1\\\a-1\end{matrix}\right.\right)dt\right\}
\\
=\frac{1}{(n-1)!}\left\{\sum_{j=0}^{n-1}\binom{n-1}{j}(-x)^{n-1-j}\int_0^1t^jG^{p,0}_{p,p}\left(t\left\vert\begin{matrix}\b-1\\\a-1\end{matrix}\right.\right)dt\right\}
\\
=\frac{\Gamma(\a)(-x)^{n-1}}{\Gamma(\b)(n-1)!}\sum_{j=0}^{n-1}\frac{(-n+1)_j(\a)_j}{(\b)_jj!}x^{-j}
=\frac{(-x)^{n-1}\Gamma(\a)}{(n-1)!\Gamma(\b)}{}_{p+1}F_p\left.\!\left(\begin{matrix}-n+1,\a\\\b\end{matrix}\right\vert 1/x\right),
\end{multline*}
where the Mellin transform of Meijer's $G$ function \cite[(16)]{KLMeijer1} has been used in the pre-ultimate equality. The positivity
restrictions $\Re(\a),\Re(\psi)>0$ can now be removed by analytic continuation.
\end{proof}

The above proposition leads immediately to the next statement.
\begin{propos}\label{lm:Gn-extendeddef}
The function $\widehat{G}_n(t)/\Gamma(\a)$  is entire in each component of $\a$ \emph{(}all apparent singularities are removable\emph{)}. If $a_{i}=-m_i$, $m_i\in\N_0$, for $i=1,\ldots,r$, $r\le{p}$, then
\begin{equation}\label{GnoverGammas}
\frac{\widehat{G}_n(x)}{\Gamma(\a)}=\frac{x^{n-1}}{(n-1)!\Gamma(\b)}{}_{p+1}F_p\left.\!\left(\begin{matrix}-n+1,-\m,\a'\\\b\end{matrix}\right\vert 1/x\right),
\end{equation}
where $\m=(m_1,\ldots,m_r)$  and $\a=(-\m,\a')$.
\end{propos}
\begin{proof} The only potential singularities of $\a\to\widehat{G}_n(t)/\Gamma(\a)$ are those points where some or all the components of $\a$ are non-positive integers,
since these points violate the separation condition necessary for existence of the contour defining $\widehat{G}_n(t)$. Suppose $\a=(-\m,\a')$, where $\m=(m_1,\ldots,m_r)$
are nonnegative integers and no component of $\a'$ is equal to a non-positive integer. Using this notation we need to calculate
$$
\lim\limits_{\hm{\epsilon}\to0}\frac{1}{\Gamma(-\m-\hm{\epsilon})\Gamma(\a')}
G^{p,1}_{p+1,p+1}\left(t\left|\begin{array}{l}\!\!n,\b+n-1\!\!\\\!\!-\m-\hm{\epsilon}+n-1,\a'+n-1,0\!\!\end{array}\right.\right),
$$
where $\hm{\epsilon}=(\epsilon_1,\ldots,\epsilon_r)$. Dividing (\ref{eq:GGF}) by $\Gamma(\a)$ and taking the limit we get (\ref{GnoverGammas})
which shows that all singularities are indeed removable.
\end{proof}

Before we turn to the next proposition we need to recall some properties of the Meijer-N{\o}rlund function $G^{p,0}_{p,p}$ elaborated in \cite[section~2]{KLMeijer1} and
\cite{KPSIGMA}.  First, we will need N{\o}rlund's expansion
\begin{equation}\label{eq:Norlund}
G^{p,0}_{p,p}\!\left(\!z~\vline\begin{array}{l}\b\\\a\end{array}\!\!\right)=\frac{z^{a_k}(1-z)^{\psi-1}}{\Gamma(\psi)}
\sum\limits_{j=0}^{\infty}\frac{g_j(\a_{[k]};\b)}{(\psi)_j}(1-z)^j,~~~k=1,2,\ldots,p,
\end{equation}
which holds in the disk $|1-z|<1$ for all $-\psi=-\sum_{i=1}^{p}(b_i-a_i)\notin\N_0$ and each $k=1,2,\ldots,p$.   The coefficients  $g_n(\a_{[k]};\b)$ are given by
\cite[(1.28), (2.7), (2.11)]{Norlund}:
\begin{equation}\label{eq:Norlund-explicit}
g_j(\a_{[p]};\b)=\sum\limits_{0\leq{j_{1}}\leq{j_{2}}\leq\cdots\leq{j_{p-2}}\leq{j}}
\prod\limits_{m=1}^{p-1}\frac{(\psi_m+j_{m-1})_{j_{m}-j_{m-1}}}{(j_{m}-j_{m-1})!}(b_{m+1}-a_{m})_{j_{m}-j_{m-1}},
\end{equation}
where $\psi_m=\sum_{i=1}^{m}(b_i-a_i)$, $j_0=0$, $j_{p-1}=j$. The coefficient $g_j(\a_{[k]};\b)$ is obtained from $g_j(\a_{[p]};\b)$
by exchanging $a_p$ and $a_k$.  These coefficients satisfy two different recurrence relations (in $p$ and $j$)
also discovered by N{\o}rlund. Details can be found in \cite[section~2.2]{KPSIGMA}.
Taking limit $\psi\to-l$, $l\in\N_0$ in (\ref{eq:Norlund}) we obtain
\begin{equation}\label{eq:Norlund1}
G^{p,0}_{p,p}\!\left(\!z~\vline\begin{array}{l}\b\\\a\end{array}\!\!\right)=z^{a_k}
\sum\limits_{j=0}^{\infty}\frac{g_{j+l+1}(\a_{[k]},\b)}{j!}(1-z)^j,~~~k=1,2,\ldots,p,
\end{equation}
where $\psi=-l$, $l\in\N_0$ (see \cite[formula~(1.34)]{Norlund}). Hence, $G^{p,0}_{p,p}$ is analytic in the neighborhood of $z=1$ for non-positive integer values of $\psi$.
The Mellin transform of $G^{p,0}_{p,p}$ exists if either  $\Re(\psi)>0$ or $\psi=-m\in\N_0$. In the former case
\begin{equation}\label{eq:GMellin}
\int\limits_{0}^{\infty}x^{s-1}G^{p,0}_{p,p}\!\left(\!x~\vline\begin{array}{l}\b\\\a\end{array}\!\!\right)dx
=\int\limits_{0}^{1}x^{s-1}G^{p,0}_{p,p}\!\left(\!x~\vline\begin{array}{l}\b\\\a\end{array}\!\!\right)dx
=\frac{\Gamma(\a+s)}{\Gamma(\b+s)}
\end{equation}
is valid in the intersection of the half-planes $\Re(s+a_i)>0$, $i=1,\ldots,p$.  If $\psi=-m\in\N_0$, then
\begin{equation}\label{eq:GMellinNorlund}
\int\limits_{0}^{\infty}x^{s-1}G^{p,0}_{p,p}\!\left(\!x~\vline\begin{array}{l}\b\\\a\end{array}\!\!\right)dx
=\int\limits_{0}^{1}x^{s-1}G^{p,0}_{p,p}\!\left(\!x~\vline\begin{array}{l}\b\\\a\end{array}\!\!\right)dx
=\frac{\Gamma(\a+s)}{\Gamma(\b+s)}-q(s)
\end{equation}
in the same half-plane, where $q(s)$ is a polynomial of degree $m$ given by
\begin{equation}\label{eq:q-polynomial}
q(s)=\sum\limits_{j=0}^{m}g_{m-j}(\a_{[k]};\b)(s+a_k-j)_j,~~~k=1,2,\ldots,p.
\end{equation}
The coefficients $g_{i}(\a_{[k]};\b)$ depend on $k$.  The resulting polynomial $q(s)$, however, is the same for each $k$.
Given a nonnegative integer $k$ suppose that $\Re(\psi)>-k$ and $\Re(\a)>0$. Then we have
\begin{equation}\label{eq:intG1-x}
\int\limits_{0}^{1}G^{p,0}_{p,p}\!\left(\!x~\vline\begin{array}{l}\b-1\\\a-1\end{array}\!\!\right)(1-x)^kdx
={\Gamma(\a)\over\Gamma(\b)}{}_{p+1}F_p\left.\left(\begin{matrix}-k, \a\\ \b\end{matrix}\right\vert 1\right).
\end{equation}

We will need the asymptotic properties of $\widehat{G}_n$ summarized in the next two propositions.

\begin{propos}\label{lm:Gn-asymp1}
Suppose that $n\in\N$, $\a$ and $\b$ are arbitrary complex vectors.  If $\Re(\psi)+n-1>0$ or if $\a$ contains non-positive integers, then
\begin{equation}\label{eq:hatG1first}
\frac{\widehat{G}_n(x)}{\Gamma(\a)}=\frac{{}_{p+1}F_{p}(-n+1,\a;\b;1)}{(n-1)!\Gamma(\b)}+o(1)~\text{as}~x\to1.
\end{equation}
If $\psi=-m$ with integer $m\ge{n-1}$ and $-\a\notin\N_0$, then
\begin{equation}\label{eq:hatG1second}
\frac{\widehat{G}_n(x)}{\Gamma(\a)}=\frac{{}_{p+1}F_{p}(-n+1,\a;\b;1)}{\Gamma(\b)(n-1)!}
-\frac{1}{\Gamma(\a)}\sum\limits_{j=0}^{n-1}\frac{(-1)^jq(j)}{(n-1-j)!j!}+o(1)~\text{as}~x\to1,
\end{equation}
where $q(\cdot)$ is given in \emph{(\ref{eq:q-polynomial})}. If $\Re(\psi)+n-1<0$, $-\psi\notin\N_0$, and $-\a\notin\N_0$, then
\begin{equation}\label{eq:hatG1third}
\frac{\widehat{G}_n(x)}{\Gamma(\a)}=\frac{(-1)^{n}(1-x)^{\psi+n-1}}{\Gamma(\a)\Gamma(\psi+n)}(1+o(1))~\text{as}~x\to1.
\end{equation}
\end{propos}
\begin{proof} Indeed, according to (\ref{eq:Norlund})
$$
\widetilde{G}_n(x)=\frac{x^{\widetilde{a}_k}(1-x)^{\psi+n-1}}{\Gamma(\psi+n)}
\sum\limits_{j=0}^{\infty}\frac{g_j(\widetilde{\a}_{[k]};\widetilde{\b})}{(\psi+n)_j}(1-x)^j,~~~k=1,2,\ldots,p+1,
$$
where $\widetilde{\a}=(\a+n-1,0)$, $\widetilde{\b}=(\b+n-1,n)$. The asymptotic relation (\ref{eq:hatG1first}) for $\Re(\psi)+n-1>0$ as well as formula (\ref{eq:hatG1third})
follow by substituting the above formula into (\ref{eq:GGF}) and letting $x\to1$.  In deducing  (\ref{eq:hatG1third}) we also used
that $g_0(\widetilde{\a}_{[k]};\widetilde{\b})=1$.  If $\a$ contains non-positive integers then (\ref{eq:hatG1first})  follows directly from (\ref{GnoverGammas}).
Finally, assume that $\psi=-m$ with integer $m\ge{n-1}$ and $\a$ does not contain non-positive integers. Then by (\ref{eq:hatGnG0}) and (\ref{eq:GMellinNorlund})
$$
\widehat{G}_n(1)=\frac{\int_0^1G_0(x)(t-x)^{n-1}dx}{(n-1)!}={\Gamma(\a){}_{p+1}F_{p}(-n+1,\a;\b;1)\over\Gamma(\b)(n-1)!}-\frac{1}{(n-1)!}\sum\limits_{j=0}^{n-1}(-1)^j\binom{n-1}{j}q(j),
$$
which is a rewriting of (\ref{eq:hatG1second}).
\end{proof}

Note that the poles of the numerator of the integrand
$$
t^{-s}\frac{\Gamma(\a+n-1+s)\Gamma(1-n\!-\!s)}{\Gamma(\b+n-1+s)\Gamma(1-s)}
$$
in the definition of $G^{p,1}_{p+1,p+1}(t)$ may cancel out with the poles of the denominator.   Suppose that $b_k=a_i+l$ for some $k=1,\ldots,p$ and $l\in\Z$.
If $l\leq0$, then all the poles of the function $\Gamma(a_i+n-1+s)$ cancel out with the poles of $\Gamma(b_k+n-1+s)$. We will call the index $i$
and the corresponding component $a_i$ normal if at least one pole of $\Gamma(a_i+n-1+s)$ does not cancel (if such pole is single then it is necessarily the rightmost one).
We say that $\a$ is normal if all its components are normal. In general situation we can ''normalize'' $\a$ by deleting the exceptional ($=$ not normal) components.

\begin{propos}\label{lm:Gn-asymp}
Suppose that $\a\in\C^{p'}$ is normal or normalized and $-\a\notin\N_0$. Set
\begin{equation}\label{eq:aAdefined}
a'=\min(\Re(a_1),\Re(a_2),\ldots,\Re(a_{p'})),~~~\mathcal{A}=\{a_i: \Re(a_i)=a'\}.
\end{equation}
\emph{(}$\mathcal{A}$  is generally a multiset, i.e. it may contain repeated elements.\emph{)}
Write $r\in\N$ for the maximal multiplicity among the elements of $\mathcal{A}$ and $\widehat{a}_1,\ldots,\widehat{a}_l$
for the distinct elements of $\mathcal{A}$ each having multiplicity $r$. Then
\begin{equation}\label{eq:Gnasymp0}
\frac{\widehat{G}_n(x)}{\Gamma(\a)}=\sum\limits_{k=1}^{l}\alpha_k x^{\widehat{a}_k+n-1}[\log(1/x)]^{r-1}[1+o(1)]~~\text{as}~x\to0,
\end{equation}
where
\begin{equation}\label{eq:alphak}
\alpha_k=\frac{\prod_{a_i\ne\widehat{a}_k}\Gamma(a_i-\widehat{a}_k)}
{(r-1)!\Gamma(\a)(\widehat{a}_k)_n\prod_{i=1}^p\Gamma(b_i-\widehat{a}_k)}.
\end{equation}
If the normalized vector $\a$ does contain non-positive integers, i.e. $\a=(-\m,\tilde{\a})$ with $\m\in\N_0^j$ and
 $\tilde{\a}\in\C^{p'-j}$, $-\tilde{\a}\notin\N_0$, then
\begin{equation}\label{eq:Gnasymp0aint}
\frac{\widehat{G}_n(x)}{\Gamma(\a)}=x^{n-1-m}\left[\frac{(-1)^m(n-m)_m(\a)_m}{\Gamma(\b+m)(n-1)!m!}+\O(x)\right]~~\text{as}~x\to0,
\end{equation}
where $m=\min(m_1,\ldots,m_j,n-1)$.
\end{propos}
\begin{proof}
The asymptotic approximation as $x\to0$ for the general Fox's $H$ function, of which Meijer's $G$ function is a particular case, is given in \cite[Theorem~1.5]{KilSaig}.
However, the computation of the constant in \cite[formula (1.4.6)]{KilSaig} seems to contain an error, so we redo this computation here. The result in \cite[Theorem~1.5]{KilSaig}
also  excludes the case when $\a$ contains non-positive integers. If $\a$ does non contain non-positive integer components, we see from \cite[Theorem~1.5]{KilSaig} that
the asymptotics of $\widehat{G}_n(x)$ as $x\to0$ is determined by the rightmost poles of the integrand having maximal multiplicity $r$,
i.e by the numbers $\widehat{a}_1,\ldots,\widehat{a}_l$.  Let us consider the contribution of the residue at the pole at $s=1-n-\widehat{a}_1$ assuming
that $\widehat{a}_1$ is not a non-positive integer. From definition (\ref{eq:Gn-defined}) we have:
\begin{multline*}
\res\limits_{s=1-n-\widehat{a}_1}{\Gamma(1-n-s)\prod_{i=1}^{p}\Gamma(a_i+n+s-1)\over
x^s\Gamma(1-s)\prod_{i=1}^p\Gamma(b_i+n+s-1)}={1\over(r-1)!}\times
\\
\lim_{s\to1-n-\widehat{a}_1}\left[{[(s+\widehat{a}_1+n-1)\Gamma(\widehat{a}_1+s+n-1)]^r\Gamma(1-n-s)\prod_{a_i\ne\widehat{a}_1}\Gamma(a_i+n+s-1)\over
\prod_{i=1}^p\Gamma(b_i+n+s-1)\Gamma(1-s)x^s}\right]^{(r-1)}\!\!\!\!\!\!\!.
\end{multline*}
Using the straightforward relations
$$
(s\!+\!\widehat{a}_1\!+\!n\!-\!1)\Gamma(\widehat{a}_1\!+\!s\!+\!n\!-\!1)=\Gamma(\widehat{a}_1\!+\!s\!+\!n),
$$
$$
{\Gamma(1-n-s)\over\Gamma(1-s)}={(-1)^n\over(s)_n},~ \frac{\partial^{r-1}}{\partial{s}^{r-1}}x^{-s}=x^{-s}(-\log{x})^{r-1}
$$
and the fact that the above $(r-1)$-th derivative in the definition of the residue  has the form
$$
\lbrace f(s)x^{-s}\rbrace^{(r-1)}=\lbrace f(s)\rbrace(x^{-s})^{(r-1)}+\O\left((x^{-s})^{(r-2)}\right) \hskip 1cm \text{as}~x\to0,
$$
we find that
\begin{multline*}
G^{p,1}_{p+1,p+1}\left(x\left|\begin{array}{l}\!\!n,\b+n-1\!\!\\\!\!\a+n-1,0\!\!\end{array}\right.\right)
={\prod_{a_i\ne\widehat{a}_1}\Gamma(a_i-\widehat{a}_1)(-1)^{n}\over
(r-1)!\prod_{i=1}^p\Gamma(b_i-\widehat{a}_1)(1-n-\widehat{a}_1)_n}x^{\widehat{a}_1+n-1}[\log(1/x)]^{r-1}
\\[5pt]
+\O(x^{\widehat{a}_1+n-1}\log^{r-2}(x))+\text{contributions from other poles}.
\end{multline*}
Finally applying $(1-n-\widehat{a}_1)_n=(-1)^n(\widehat{a}_1)_n$ we obtain
$$
G^{p,1}_{p+1,p+1}\left(x\left|\begin{array}{l}\!\!n,\b+n-1\!\!\\\!\!\a+n-1,0\!\!\end{array}\right.\right)
\!\!=\!\!\tilde{\alpha}_1 x^{\widehat{a}_1+n-1}[\log(1/x)]^{r-1}+\O(x^{\widehat{a}_1+n-1}\log^{r-2}(x))+\text{residues at other poles},
$$
with
$$
\tilde{\alpha}_1={\prod_{a_i\ne\widehat{a}_1}\Gamma(a_i-\widehat{a}_1)\over(r-1)!(\widehat{a}_1)_n\prod_{i=1}^p\Gamma(b_i-\widehat{a}_1)}.
$$
Adding up similar contributions from the poles at $s=1-n-\widehat{a}_i$, $i=1,\ldots,l$ , and dividing through by $\Gamma(\a)$ we arrive at formula (\ref{eq:Gnasymp0})
for $-\widehat{a}_j\notin\N_0$ for $j=1,\ldots,l$. If the pole at  $s=1-n-\widehat{a}_1$ is simple, then $r=1$ and the above calculation simplifies,
namely $\O(x^{\widehat{a}_1+n-1}\log^{r-2}(x))$ disappears and the principal contribution from other poles will have asymptotic
order $\O(x^{\tilde{a}_2+n-1})$, where $\tilde{a}_2$ is the element with the second smallest real part.  This confirms that (\ref{eq:Gnasymp0})
remains valid in this case.

If $\a=(-\m,\tilde{\a})$ with $\m=(m_1,\ldots,m_j)\in\N_0^j$, then the result  (\ref{eq:Gnasymp0aint}) follows immediately from (\ref{GnoverGammas})
in view of the identity $(1-n)_m=(-1)^m(n-m)_m$.
\end{proof}

\textbf{Remark.}  In what follows the case of real parameters will play a special role. In this case we necessarily have $l=1$ and $\widehat{a}_1=\min_i(a_i)$.
If $\a$ does not contains non-positive integers, then the sign of $\widehat{G}_n(x)/\Gamma(\a)$ in the neighborhood of $x=0$ is determined by the sign of the real nonzero
constant $\alpha_1$ from (\ref{eq:alphak}). In this case, define $\eta\in\{0,1\}$ implicitly by
\begin{subequations}
\label{eq:et}
\begin{equation}\label{eq:et1}
(-1)^{\eta}=\mathrm{sgn}(\alpha_1)=\mathrm{sgn}\left[\frac{1}{\Gamma(\a)(\widehat{a}_1)_n\prod_{i=1}^p\Gamma(b_i-\widehat{a}_1)}\right],
\end{equation}
where $b_i-\widehat{a}_1$ can not take non-positive integer values as we assume $\a$ to be normalized as explained before Proposition~\ref{lm:Gn-asymp}.
If $\a$ contains non-positive integers then the sign of $\widehat{G}_n(x)/\Gamma(\a)$ in the neighborhood of $x=0$ is determined by the
sign of the constant in (\ref{eq:Gnasymp0aint}), and we define $\eta\in\{0,1\}$ by
\begin{equation}\label{eq:et2}
(-1)^{\eta}=\mathrm{sgn}\left[\frac{(-1)^m(\a)_m}{\Gamma(\b+m)}\right].
\end{equation}
\end{subequations}
Note that formulas (\ref{eq:et}) imply that $(-1)^{\eta}\widehat{G}_n(x)/\Gamma(\a)$ is positive in the neighborhood of $x=0$
for all real vectors $\a$ and $\b$ and the number $\eta$ is independent of $n$ once $n+\widehat{a}_1>0$.

\smallskip

Next, we show that $\widehat{G}_n(x)$ and $\widehat{G}_m(t)$ are related by the Riemann-Liouville fractional integral.
\begin{propos}\label{lm:Gnfracintegral}
Suppose $n>m\ge0$ are integers and $\a$, $\b$ are arbitrary complex vectors satisfying $m+\Re(\a)>0$. Then
\begin{equation}\label{eq:GnGm}
\frac{\widehat{G}_n(x)}{\Gamma(\a)}=\frac{1}{(n-m-1)!}\int_0^x\frac{\widehat{G}_m(t)}{\Gamma(\a)}(x-t)^{n-m-1}dt.
\end{equation}
\end{propos}
\begin{proof}
If $\a$ does not contain non-positive integer components, then the claim follows from a particular case
of \cite[2.24.2.2]{PBM3}. If $\a$ contains such components formula (\ref{eq:GnGm}) can either be justified by analytic
continuation in $\a$ in view of Proposition~\ref{lm:Gn-extendeddef} or confirmed directly by termwise integration
of (\ref{GnoverGammas}).  The convergence of the integral under the specified conditions follows from Proposition~\ref{lm:Gn-asymp} (for $m>0$)
and \cite[Property~5]{KLMeijer1} (for $m=0$).
\end{proof}
Proposition~\ref{lm:Gnfracintegral} implies, in particular,  that $\widehat{G}_{m+1}'(x)=\widehat{G}_{m}(x)$
and $\widehat{G}_{m}(0)=0$ for $m+\Re(\a)>0$.

\smallskip

In the next proposition we compute the moments of $\widehat{G}_n(t)$.
\begin{propos}\label{lm:Gnmoments}
Suppose $\N_0\ni{n}>-\min(\Re(\psi),\Re(\a))$. Then the following formulas hold\emph{:}
\begin{equation}\label{eq:gk}
m_k\!=\!\!\!\int\limits_0^1\!\!\widehat{G}_n(t)t^kdt\!=\!\frac{\Gamma(\a)}{\Gamma(\b)(n\!+\!k)(n\!-\!1)!}{}_{p+2}F_{p+1}\left.\!\!\left(\!\begin{matrix}-n-k,-n+1,\a\\-n-k+1,\b\end{matrix}\right\vert 1\!\right)\!+\!\frac{(-1)^n\Gamma(\a+n+k)k!}{\Gamma(\b+n+k)(n+k)!},
\end{equation}
\begin{equation}\label{eq:bargk}
\hat{m}_k=\int_0^1\widehat{G}_n(t)(1-t)^kdt=\frac{\Gamma(\a)k!}{\Gamma(\b)(n+k)!}{}_{p+1}F_{p}\left.\left(\begin{matrix}-n-k,\a\\\b\end{matrix}\right\vert 1\right)
\end{equation}
for $k\in\N_0$ and $1/(-1)!=0$.  Moreover,
define $\Delta{m_k}=m_{k+1}-m_{k}$, $\Delta^r{m_k}=\Delta(\Delta^{r-1}{m_k})$. Then, for $k,r\in\N_0$,
\begin{equation}\label{eq:Deltamgk}
\Delta^r{m_k}=\Delta^k{\hat{m}_r}=\int_0^1\widehat{G}_n(t)(1-t)^rt^kdt
=\frac{\Gamma(\a)}{\Gamma(\b)}\sum\limits_{j=0}^{k}\binom{k}{j}\frac{(-1)^{j}(r+j)!}{(n+r+j)!}
{}_{p+1}F_p\left.\left(\begin{matrix}-n-r-j,\a\\ \b\end{matrix}\right\vert 1\right).
\end{equation}
\end{propos}
\textbf{Remark.}  Formulas (\ref{eq:gk})-(\ref{eq:Deltamgk}) remain valid if $\a$ contains non-positive integers
after division by $\Gamma(\a)$ in view of Proposition~\ref{lm:Gn-extendeddef}.

\begin{proof} Assume for a moment that $\Re(\a)>0$ and $\Re(\psi)>0$.  Then we have
\begin{multline*}
\int_0^1\widehat{G}_n(t)t^kdt=\frac{1}{(n-1)!}\int_0^1t^kdt\int_0^tG^{p,0}_{p,p}\left(x\left\vert\begin{matrix}\b-1\\\a-1\end{matrix}\right.\right)(t-x)^{n-1}dx
\\
=\frac{1}{(n-1)!}\int_0^1G^{p,0}_{p,p}\left(x\left\vert\begin{matrix}\b-1\\\a-1\end{matrix}\right.\right)dx\int_x^1(t-x)^{n-1}t^kdt.
\end{multline*}
The inner integral is computed by an application of the binomial expansion to $(t-x)^{n-1}$:
\begin{multline*}
\int_x^1(t-x)^{n-1}t^kdt=\int_x^1t^kdt\sum\limits_{l=0}^{n-1}\binom{n-1}{l}t^l(-x)^{n-1-l}
=\sum\limits_{l=0}^{n-1}\binom{n-1}{l}(-x)^{n-1-l}\int_x^1t^{k+l}dt
\\
=\sum\limits_{l=0}^{n-1}(-1)^{n-1-l}\binom{n-1}{l}\frac{x^{n-1-l}-x^{n+k}}{k+l+1}
=\frac{1}{n+k}{}_2F_1\left.\left(\begin{matrix}-n-k,-n+1\\-n-k+1\end{matrix}\right\vert x\right)+\frac{(-1)^nx^{n+k}}{n\binom{n+k}{k}}.
\end{multline*}
The last equality can be verified by comparing the coefficients at equal powers of $x$.  Substituting this expression into
the formula above, integrating termwise and applying (\ref{eq:GMellin}) we arrive at (\ref{eq:gk}).
To verify (\ref{eq:bargk})  write
$$
\hat{m}_k=\int_0^1\widehat{G}_n(t)(1-t)^kdt=\frac{1}{(n-1)!}\int_0^1(1-t)^kdt\int_0^tG^{p,0}_{p,p}\left(x\left\vert\begin{matrix}\b-1\\\a-1\end{matrix}\right.\right)(t-x)^{n-1}dx
$$
and repeat the above steps with (\ref{eq:intG1-x}) in place of (\ref{eq:GMellin}) to obtain (\ref{eq:bargk}).  Finally,
to get (\ref{eq:Deltamgk}) calculate
$$
\int_0^1(1-t)^r\widehat{G}_n(t)t^kdt=\frac{1}{(n-1)!}\int_0^1G^{p,0}_{p,p}\left(x\left\vert\begin{matrix}\b-1\\\a-1\end{matrix}\right.\right)dx\int_x^1(t-x)^{n-1}t^k(1-t)^rdt.
$$
Now, using
$$
\int_x^1(t-x)^{n-1}(1-t)^ldt=\frac{(1-x)^{n+l}l!(n-1)!}{(n+l)!},
$$
we have
\begin{multline*}
\int_x^1(t-x)^{n-1}t^k(1-t)^rdt=\int_x^1(t-x)^{n-1}(1-(1-t))^k(1-t)^rdt
\\
=\sum\limits_{j=0}^{k}(-1)^{j}\binom{k}{j}\int_x^1(t-x)^{n-1}(1-t)^{r+j}dt
=\sum\limits_{j=0}^{k}(-1)^{j}\binom{k}{j}\frac{(1-x)^{n+r+j}(r+j)!(n-1)!}{(n+r+j)!},
\end{multline*}
so that
$$
\Delta^r{m_k}=\Delta^k{\hat{m}_r}=\int_0^1(1-t)^r\widehat{G}_n(t)t^kdt=\frac{\Gamma(\a)}{\Gamma(\b)}\sum\limits_{j=0}^{k}(-1)^{j}\binom{k}{j}\frac{(r+j)!}{(n+r+j)!}
{}_{p+1}F_p\left.\left(\begin{matrix}-n-r-j, \a\\ \b\end{matrix}\right\vert 1\right).
$$
The restrictions $\Re(\a)>0$,  $\Re(\psi)>0$ can now be removed by analytic continuation.
\end{proof}

\begin{corollary}
For any natural $n$ and nonnegative integer $k$ the next identity holds\emph{:}
$$
\sum\limits_{j=0}^{k}\binom{k}{j}\frac{(-1)^{j}j!}{(n+j)!}
{}_{p+1}F_p\left.\!\!\left(\!\begin{matrix}-n-j, \a\\ \b\end{matrix}\right\vert 1\!\right)\!=
\!\frac{1}{(n+k)\Gamma(n)}{}_{p+2}F_{p+1}\left.\!\!\left(\!\begin{matrix}-n-k,-n+1,\a\\-n-k+1,\b\end{matrix}\right\vert 1\!\right)
+\frac{(-1)^n(\a)_{n+k}k!}{(\b)_{n+k}(n+k)!}.
$$
\end{corollary}
\begin{proof} The result follows on comparing (\ref{eq:gk}) with the $r=0$ case of (\ref{eq:Deltamgk}).
Note that the claimed identity is not contained in \cite{PBM3} but can be deduced from \cite[15.3.2.12]{PBM3}
by the appropriate limit transition.
\end{proof}

Formulas for the moments derived in Proposition~\ref{lm:Gnmoments}  play the key role in the following computation of
 the hypergeometric transform of  $\widehat{G}_n(t)$.

\begin{propos}\label{lm:Gnhypergeomint}
Suppose that $u\le{s+1}$ are nonnegative integers, $\c\in\C^{u}$, $\d\in\C^{s}$, $\Re(\psi)>-n$ and $\Re(\a)>-n$,
where  $\psi$ is defined in \textup{(\ref{eq:a-psi-defined})}. Then
\begin{equation}\label{eq:Gnhypergeomint}
\frac{1}{\Gamma(\a)}\int_0^1\!\!{}_uF_{s}\left.\!\left(\begin{matrix}\c\\\d\end{matrix}\:\right\vert -\!zt\right)\widehat{G}_n(t)dt
=\frac{1}{\Gamma(\b)}\sum_{j=0}^{\infty}\frac{z^j(\c)_j}{(n+j)!(\d)_j}
{}_uF_{s}\!\left.\left(\begin{matrix}\c+j\\\d+j\end{matrix}\right\vert -\!z\!\right)
{}_{p+1}F_p\left.\!\left(\begin{matrix}-n-j,\a\\\b\end{matrix}\right\vert 1\!\right),
\end{equation}
where the series on the right converges for all $z\in\C$ if $u\le{s}$ and in the half-plane $\Re(z)>-1/2$ if $u=s+1$.
Therefore, for $u=s+1$ the integral on the left hand side is an explicit representation of the analytic continuation to the cut plane
$\C\!\setminus\!(-\infty,-1]$ of the function defined by the right hand side.
\end{propos}
\begin{proof}  Assume first that $\Re(\a)>0$.  Then using (\ref{eq:intG1-x}) and the $r=0$ case of (\ref{eq:Deltamgk}) we obtain
by termwise integration and interchange of the order of summations:
\begin{multline*}
\frac{1}{\Gamma(\a)}\int_0^1{}_uF_{s}\left.\!\left(\begin{matrix}\c\\\d\end{matrix}\right\vert -\!zt\right)\widehat{G}_n(t)dt
=\frac{1}{\Gamma(\b)}\sum_{k=0}^{\infty}\frac{(\c)_k(-z)^k}{(\d)_kk!}\sum\limits_{j=0}^{k}\binom{k}{j}\frac{(-1)^{j}j!}{(n+j)!}
{}_{p+1}F_p\left.\left(\begin{matrix}-n-j,\a\\ \b\end{matrix}\right\vert 1\right)
\\
=\frac{1}{\Gamma(\b)}\sum\limits_{j=0}^{\infty}\frac{(-1)^{j}j!}{(n+j)!}
{}_{p+1}F_p\left.\left(\begin{matrix}-n-j,\a\\ \b\end{matrix}\right\vert 1\right)
\sum_{k=j}^{\infty}\binom{k}{j}\frac{(\c)_k(-z)^k}{(\d)_kk!}
\\
=\frac{1}{\Gamma(\b)}\sum_{j=0}^{\infty}\frac{z^j(\c)_j}{(n+j)!(\d)_j}
{}_{p+1}F_p\left.\!\left(\begin{matrix}-n-j,\a\\\b\end{matrix}\right\vert 1\!\right)
\sum_{l=0}^{\infty}\frac{(\c+j)_l(-z)^{l}}{(\d+j)_ll!}.
\end{multline*}
The claims regarding the convergence domains are justified by the following formulas due to Krottnerus.
If $u=s+1$, then by \cite[7.3(3)]{LukeBook} for $z\in\C\!\setminus\!(-\infty,-1]$
$$
{}_{s+1}F_{s}\!\left.\left(\begin{matrix}\c+j\\\d+j\end{matrix}\right\vert -\!z\!\right)=(1+z)^{\nu-j}\left(1+\O(j^{-1})\right)
~\text{as}~j\to\infty~\text{with}~\nu=\sum_{i=1}^{s}d_i-\sum_{i=1}^{s+1}c_i.
$$
If $u=s$, then by  \cite[7.3(4)]{LukeBook} for all $z\in\C$,
$$
{}_{s}F_{s}\!\left.\left(\begin{matrix}\c+j\\\d+j\end{matrix}\right\vert -\!z\!\right)=e^{-z}\left(1+\O(j^{-1})\right)
~\text{as}~j\to\infty.
$$
Finally, if $u<s$ by \cite[7.3(5)]{LukeBook} for all $z\in\C$,
$$
{}_{u}F_{s}\!\left.\left(\begin{matrix}\c+j\\\d+j\end{matrix}\right\vert -\!z\!\right)=1+\O(j^{u-s})~\text{as}~j\to\infty.
$$
It remains to note that ${}_{p+1}F_p\!\left(-n-j,\a;\b\vert 1\!\right)$ cannot grow faster than polynomially
by (\ref{eq:bargk}) and $|z/(z+1)|<1$ is equivalent to $\Re(z)>-1/2$.  As both sides of (\ref{eq:Gnhypergeomint}) are analytic in each $a_i$
in the domain $\Re(a_i)>-n$, the formula holds  in the region stated in the conclusions of the proposition
in view of Proposition~\ref{lm:Gn-extendeddef}.
\end{proof}

Particular cases of Proposition~\ref{lm:Gnhypergeomint} lead to the formulas for the generalized Stieltjes, Laplace and (slightly modified) Hankel transform of $\widehat{G}_n(t)$.

\begin{corollary}\label{cr:transforms}
Suppose $\Re(\psi)>-n$ and $\Re(\a)>-n$. Then for any $\sigma\in\C$,
\begin{equation}\label{eq:Gnpowerint}
\frac{1}{\Gamma(\a)}\int_0^1\frac{\widehat{G}_n(t)}{(1+zt)^{\sigma}}dt
=\frac{1}{\Gamma(\b)(1+z)^{\sigma}}\sum_{j=0}^{\infty}\frac{(\sigma)_jz^j}{(n+j)!(1+z)^j}
{}_{p+1}F_p\left.\left(\begin{matrix}\a,-n-j\\ \b\end{matrix}\right\vert 1\right),
\end{equation}
where the series on the right hand side converges in the half-plane $\Re(z)>-1/2$.
Further,
\begin{equation}\label{eq:Gnexp}
\frac{1}{\Gamma(\a)}\int_0^1\widehat{G}_n(t)e^{-zt}dt=\frac{e^{-z}}{\Gamma(\b)}\sum_{j=0}^{\infty}\frac{z^j}{(n+j)!}
 {}_{p+1}F_p\left(\left.\begin{matrix}\a,-n-j\\\b\end{matrix}\right\vert 1\right)
\end{equation}
and
\begin{equation}\label{eq:GnBessel}
\frac{1}{\Gamma(\a)}\int_0^1\widehat{G}_n(t){}_0F_{1}(-;\nu;-zt)dt=\frac{1}{\Gamma(\b)}\sum\limits_{j=0}^{\infty}\frac{z^j{}_0F_{1}(-;\nu+j;-z)}{(n+j)!(\nu)_{j}}
{}_{p+1}F_p\left.\left(\begin{matrix}-n-j,\a\\ \b\end{matrix}\right\vert 1\right)
\end{equation}
for all $z\in\C$ and $\nu\in\C$.
\end{corollary}

\begin{corollary}\label{cr:summation}
The following summation formula holds:
\begin{equation}\label{eq:FFsummation}
{}_{u+p}F_{s+p}\!\left.\left(\begin{matrix}\a,\c\\\b,\d\end{matrix}\right\vert -\!z\!\right)
=\sum_{j=0}^{\infty}\frac{(\c)_jz^j}{(\d)_jj!}
{}_uF_{s}\!\left.\left(\begin{matrix}\c+j\\\d+j\end{matrix}\right\vert -\!z\!\right)
{}_{p+1}F_p\left.\!\left(\begin{matrix}-j,\a\\\b\end{matrix}\right\vert 1\!\right),
\end{equation}
where the series converges for $z\in\C$ if $u\le{s}$ and for $\Re(z)>-1/2$ if $u=s+1$.
\end{corollary}
\begin{proof}
Take $n=0$ in (\ref{eq:Gnhypergeomint}) and apply \cite[Theorem~1]{KarpJMS}.
\end{proof}

If $u=s+1=1$ in (\ref{eq:FFsummation}) we get N{\o}rlund's formula \cite[(1.21)]{Norlund} (see also \cite[formula~6.8.1.3]{PBM3}), which generalizes  Pfaff's
transformation \cite[(2.2.6)]{AAR} for ${}_2F_1$.  Taking $u=s=0$ we obtain the  summation formula \cite[formula~6.8.1.2]{PBM3} generalizing Kummer's
transformation  \cite[(4.1.11)]{AAR} for ${}_1F_1$.  Finally, if $s=u+1=1$ we get the known summation formula \cite[formula~6.8.3.4]{PBM3}.

The next elementary lemma on sign stabilization of the Riemann-Liouville fractional integral may be of independent
interest.
\begin{lemma}\label{lm:fractional}
Suppose  $f:(0,1]\to\R$ is continuous and integrable \textup{(}possibly in improper sense\textup{)}.  If $f>0$ in some  neighborhood of zero, then
there exists $\alpha>0$ such that the Riemann-Liouville fractional integral
$$
[I_{+}^{\alpha}f](x)=\frac{1}{\Gamma(\alpha)}\int\nolimits_{0}^{x}f(t)(x-t)^{\alpha-1}dt
$$
is positive for all $x\in(0,1]$.
\end{lemma}
\begin{proof}  It follows from the conditions of the lemma that there exist $0<t_0<t_1\le1$
such that $f(t)\ge\delta$ on $[t_0,t_1]$ for some $\delta>0$. Then if $x\le{t_1}$ the claim is obviously true for all $\alpha\ge1$.
Assume that $x>t_1$ and estimate
\begin{multline*}
\int\nolimits_{0}^{x}f(t)(x-t)^{\alpha-1}dt=x^{\alpha}\int\nolimits_{0}^{1}f(xu)(1-u)^{\alpha-1}du
\\
>x^{\alpha}\int\nolimits_{t_0/x}^{t_1/x}f(xu)(1-u)^{\alpha-1}du+x^{\alpha}\int\nolimits_{t_1/x}^{1}f(xu)(1-u)^{\alpha-1}du
\\
\ge x^{\alpha}\delta\int\nolimits_{t_0/x}^{t_1/x}(1-u)^{\alpha-1}du+x^{\alpha}f(t_{\alpha})\int\nolimits_{t_1/x}^{1}(1-u)^{\alpha-1}du
\\
=\frac{{\delta}x^{\alpha}}{\alpha}\left((1-t_0/x)^{\alpha}-(1-t_1/x)^{\alpha}\right)
+\frac{x^{\alpha}f(t_{\alpha})}{\alpha}(1-t_1/x)^{\alpha}
\\
=\frac{x^{\alpha}(1-t_0/x)^{\alpha}}{\alpha}\left\{\delta+(f(t_{\alpha})-\delta)\left(\frac{x-t_1}{x-t_0}\right)^{\!\alpha}\right\},
\end{multline*}
where we applied the mean value theorem to the second integral on the second line,
so that $t_{\alpha}\in[t_1,x]$.  The second term inside the braces clearly tends to zero
as $\alpha\to\infty$ and positivity follows.
\end{proof}

\textbf{Remark.}  This lemma admits an obvious generalization as follows.
Since $I_{+}^{\alpha}f=I_{+}^{\alpha_2}I_{+}^{\alpha_1}f$ for $\alpha=\alpha_1+\alpha_2$ it is sufficient to assume
that the conditions of the lemma hold for $I_{+}^{\alpha_1}f$ for some $\alpha_1\ge0$.

\smallskip

The above lemma leads to the following statement regarding the sign stabilization of $\widehat{G}_n/\Gamma(\a)$
as $n$ grows to infinity.

\begin{propos}\label{pr:signG}
For arbitrary real vectors $\a$, $\b$ there exists $N\in\N_0$ such that
$$
\frac{(-1)^{\eta}}{\Gamma(\a)}G^{p,1}_{p+1,p+1}\left(t\left|\begin{array}{l}\!\!n,\b+n-1\!\!\\\!\!\a+n-1,0\!\!\end{array}\right.\right)>0
$$
for all $n\ge{N}$,  $t\in(0,1]$ and $\eta$ given in \emph{(\ref{eq:et})}.
\end{propos}
\begin{proof}   By Proposition~\ref{lm:Gnfracintegral} for $n=k+m$ we have
$$
\frac{1}{\Gamma(\a)}G^{p,1}_{p+1,p+1}\left(t\left|\begin{array}{l}\!\!n,\b+n-1\!\!\\\!\!\a+n-1,0\!\!\end{array}\right.\right)
=I_+^{k}\left[\frac{1}{\Gamma(\a)}G^{p,1}_{p+1,p+1}\left(\cdot\left|\begin{array}{l}\!\!m,\b+m-1\!\!\\\!\!\a+m-1,0\!\!\end{array}\right.\right)\right](t).
$$
On the other hand, according to Proposition~\ref{lm:Gn-asymp} there exists $m\in\N_0$ such that
$$
\frac{(-1)^{\eta}}{\Gamma(\a)}G^{p,1}_{p+1,p+1}\left(t\left|\begin{array}{l}\!\!m,\b+m-1\!\!\\\!\!\a+m-1,0\!\!\end{array}\right.\right)
$$
satisfies the conditions of Lemma~\ref{lm:fractional} and the claim follows.
\end{proof}

Recall that a sequence $\{f_k\}_{k\ge0}$ is completely monotonic if $(-1)^m\Delta^m{f_k}\ge0$ for all integer
$m,k\ge0$.  By the celebrated result of Hausdorff, the necessary and sufficient conditions for a sequence to be completely monotone is that it is equal to the moment sequence of a nonnegative measure supported on $[0,1]$.  In view of this fact we get the next corollary of Proposition~\ref{pr:signG}.

\begin{corollary}\label{cr:Hasdorff}
For arbitrary real vectors $\a$, $\b$ there exists $N\in\N_0$, such that for all $n\ge{N}$ both
sequences $(-1)^{\eta}m_k/\Gamma(\a)$ and $(-1)^{\eta}\hat{m}_k/\Gamma(\a)$ defined in $(\ref{eq:gk})$ and $(\ref{eq:bargk})$,
respectively, are completely monotonic.  Here $\eta$ is given in \emph{(\ref{eq:et})}.
\end{corollary}

It is then natural to formulate the next

\smallskip

\noindent\textbf{Open problem.} \emph{How to find or estimate $N$ in Proposition~\ref{pr:signG} and Corollary~\ref{cr:Hasdorff}?}

\smallskip

We conclude this section with some examples of the the function $\widehat{G}_n(t)$ for small $p$.

\smallskip

\textbf{Example~1}. Take the simplest case $p=1$.  According to \cite[8.4.49.19]{PBM3}
$$
G^{1,1}_{2,2}\left(t\left|\begin{array}{l}\!\!n,b+n-1\!\!\\\!\!a+n-1,0\!\!\end{array}\right.\right)
=\frac{t^{a+n-1}}{\Gamma(\psi)(a)_n}{_2F_1}\!\!\left(\!\!\begin{array}{c}a,1-\psi\\a+n\end{array}\!\!\vline\,t\!\right)
$$
for $0<t<1$ and all values of $a$ and $\psi=b-a$.  In particular, this function vanishes for $-\psi\in\N_0$.
Substituting the above expression in (\ref{eq:GGF}) and using (\ref{eq:Meijer2reduced}) for $G^{2,0}_{2,2}$ we
arrive at the next identity:
$$
\frac{x^{a+n-1}}{\Gamma(\psi)(a)_n}{_2F_1}\!\!\left(\!\!\!\begin{array}{c}a,1-\psi\\a+n\end{array}\!\!\vline\,x\!\right)
\!=\!\frac{(-1)^n(1-x)^{\psi+n-1}}{\Gamma(\psi+n)}{_2F_1}\!\!\left(\!\!\!\begin{array}{c}1-a,\psi\\\psi+n\end{array}\!\!\vline\,1-x\!\right)
+\frac{x^{n-1}\Gamma(a)}{(n-1)!\Gamma(b)}{_2F_1}\!\!\left(\!\!\!\begin{array}{c}a,1-n\\b\end{array}\!\!\vline\,\frac{1}{x}\!\right).
$$
Certainly, there are many other ways to prove the above formula, but it does not seem to appear in the literature in this form.

\medskip

\textbf{Example~2}. For $p=2$ and  $-\psi=a_1+a_2-b_1-b_2\notin\N_0$ we have by \cite[8.4.49.22]{PBM3}:
\begin{equation}\label{eq:Meijer2reduced}
G^{2,0}_{2,2}\!\left(\!x~\vline\begin{array}{l}b_1,b_2
\\a_1,a_2\end{array}\!\!\right)=\frac{x^{a_2}(1-x)^{\psi-1}_{+}}{\Gamma(\psi)}
{_2F_{1}}\!\left(\!\!\begin{array}{c}b_1-a_1,b_2-a_1
\\\psi\end{array}\!\!\vline\,1-x\right),
\end{equation}
where $a_1$ and $a_2$ may be interchanged on the right hand side.  If $\psi=-m$, $m\in\N_0$,
an easy calculation based on (\ref{eq:Meijer2reduced}) leads to
$$
G^{2,0}_{2,2}\!\left(\!x~\vline\begin{array}{l}b_1,b_2
\\a_1,a_2\end{array}\!\!\right)\!=\!\frac{x^{a_2}(b_1-a_1)_{m+1}(b_2-a_1)_{m+1}}{(m+1)!}
{_2F_{1}}\!\left(\!\!\begin{array}{c}b_1-a_1+m+1,b_2-a_1+m+1
\\m+2\end{array}\!\!\vline\,1-x\right)
$$
for $t\in(0,1)$. Further, using a variation of Euler's integral representation
\begin{equation}\label{eq:fractint}
\int_0^t x^{a_2-1}(1-x)^{\psi-1+k}(t-x)^{n-1}dx=\frac{\Gamma(n)\Gamma(a_2)}{\Gamma(a_2+n)}t^{a_2-1+n}
{_2F_{1}}\!\left(\!\!\begin{array}{c}a_2,1-\psi-k
\\a_2+n\end{array}\!\!\vline\,t\right),
\end{equation}
we get for $-\psi\notin\N_0$ by termwise integration:
\begin{subequations}
\begin{multline}\label{eq:G2133-1}
G^{2,1}_{3,3}\!\left(t\left|\begin{array}{l}\!\!n,\b+n-1\!\!\\\!\!\a+n-1,0\!\!\end{array}\right.\right)
=\frac{1}{(n-1)!}\int_0^t(t-x)^{n-1}G^{2,0}_{2,2}\!\left(\!x~\vline\begin{array}{l}\b-1
\\\a-1\end{array}\!\!\right)dx
\\
=\frac{t^{a_2-1+n}}{\Gamma(\psi)(a_2)_n}\sum\limits_{k=0}^{\infty}\frac{(b_1-a_1)_k(b_2-a_1)_k}{(\psi)_kk!}
{_2F_{1}}\!\left(\!\!\begin{array}{c}a_2,1-\psi-k\\a_2+n\end{array}\!\!\vline\,t\right),
\end{multline}
where we utilized (\ref{eq:Meijer2reduced}) and (\ref{eq:fractint}).  Invoking \cite[7.2(24)]{LukeBook} it is easy to see that the above series converges in the unit disk
if $\Re(a_1)<1$. Exchanging the order of summations in the last series and applying the Euler transformation \cite[(2.2.7)]{AAR} we get an alternative expression:
\begin{multline}\label{eq:G2133-2}
G^{2,1}_{3,3}\!\left(t\left|\begin{array}{l}\!\!n,\b+n-1\!\!\\\!\!\a+n-1,0\!\!\end{array}\right.\right)
\\
=\frac{t^{a_2-1+n}(1-t)^{\psi+n-1}}{\Gamma(\psi)(a_2)_n}\sum\limits_{k=0}^{\infty}\frac{(n)_k(a_2+\psi+n-1)_k}{(a_2+n)_kk!}t^k
{_3F_{2}}\!\left(\!\!\!\begin{array}{c}b_1-a_1,b_2-a_1,a_2+\psi+n+k-1\\\psi,a_2+\psi+n-1\end{array}\!\!\vline\,1-t\right)
\end{multline}
converging in the unit disk $|t|<1$ for all values of parameters.
The standard expression for $G^{2,1}_{3,3}$ is the following (see \cite[8.2.2.3]{PBM3} or \cite[16.17.2]{NIST}):
\begin{multline}\label{eq:G2133-3}
G^{2,1}_{3,3}\!\left(t\left|\begin{array}{l}\!\!n,\b+n-1\!\!\\\!\!\a+n-1,0\!\!\end{array}\right.\right)
=\frac{\Gamma(a_2-a_1)t^{a_1+n-1}}{(a_1)_n\Gamma(b_1-a_1)\Gamma(b_2-a_1)}
{_3F_{2}}\!\left(\begin{array}{c}a_1,a_1-b_1+1,a_1-b_2+1\\a_1-a_2+1,a_1+n\end{array}\!\!\vline\,t\right)
\\
+\frac{\Gamma(a_1-a_2)t^{a_2+n-1}}{(a_2)_n\Gamma(b_1-a_2)\Gamma(b_2-a_2)}
{_3F_{2}}\!\left(\begin{array}{c}a_2,a_2-b_1+1,a_2-b_2+1\\a_2-a_1+1,a_2+n\end{array}\!\!\vline\,t\right).
\end{multline}
It is valid if $a_1-a_2\notin\Z$. On the other hand we can apply (\ref{eq:GGF}) to express $G^{2,1}_{3,3}$ in terms of $G^{3,0}_{3,3}$ and ${}_3F_{2}$.
Further, applying (\ref{eq:G3033}) to express
$G^{3,0}_{3,3}$ we get the identity:
\begin{multline}\label{eq:G2133-4}
G^{2,1}_{3,3}\!\left(t\left|\begin{array}{l}\!\!n,\b+n-1\!\!\\\!\!\a+n-1,0\!\!\end{array}\right.\right)
=\frac{(-1)^n(1-t)^{\psi+n-1}}{\Gamma(\psi+n)}\sum\limits_{k=0}^{\infty}\frac{(\psi-b_1+1)_k(\psi-b_2+1)_k}{(\psi+n)_kk!}\times
\\
{_3F_{2}}\!\left(\begin{array}{c}-k,1-a_1,1-a_2\\\psi-b_1+1,\psi-b_2+1\end{array}\!\right)(1-t)^k
+\frac{t^{n-1}\Gamma(\a)}{\Gamma(\b)(n-1)!}
{_3F_{2}}\!\left(\begin{array}{c}1-n,\a\\\b\end{array}\!\!\vline\,\frac{1}{t}\right),
\end{multline}
where the argument $1$ is omitted in ${}_3F_2$ for conciseness. Employing an alternative expression \cite[(8.4.51.2)]{PBM3}
for $G^{3,0}_{3,3}$ in terms of Appell's hypergeometric function $F_3$ of two variables \cite[16.13.3]{NIST}
we get
\begin{multline}\label{eq:G2133-5}
G^{2,1}_{3,3}\!\left(t\left|\begin{array}{l}\!\!n,\b+n-1\!\!\\\!\!\a+n-1,0\!\!\end{array}\right.\right)
=\frac{t^{a_1+a_2-b_1+n-2}(1-t)^{\psi+n-1}_+}{(-1)^n\Gamma(\psi+n)}
F_3(b_1-a_2,n;b_1-a_1,b_2-1;\psi+n;1-1/t,1-t)
\\
+\frac{t^{n-1}\Gamma(\a)}{\Gamma(\b)(n-1)!}
{_3F_{2}}\!\left(\begin{array}{c}1-n,\a\\\b\end{array}\!\!\vline\,\frac{1}{t}\right).
\end{multline}
\end{subequations}
Equating the right hand sides of the formulas (\ref{eq:G2133-1})-(\ref{eq:G2133-5}) we arrive at transformation  formulas for the sums of ${}_3F_2$s and  reduction
formulas for Appell's $F_3$ function. Some of these formulas might be new.

Finally, if $-\psi=m\in\N_0$ we have
\begin{multline*}
G^{2,1}_{3,3}\!\left(t\left|\begin{array}{l}\!\!n,\b+n-1\!\!\\\!\!\a+n-1,0\!\!\end{array}\right.\right)
\\
=\frac{t^{a_2-1+n}(b_1-a_1)_{m+1}(b_2-a_1)_{m+1}}{(m+1)!(a_2)_n}\sum\limits_{k=0}^{\infty}\frac{(b_1-a_1+m+1)_k(b_2-a_1+m+1)_k}{(m+1)_kk!}
{_2F_{1}}\!\left(\begin{array}{l}a_2,-k\\a_2+n\end{array}\!\!\vline\,t\right).
\end{multline*}

\medskip

\textbf{Example~3}. For $p=3$ the coefficients in (\ref{eq:Norlund}) are given by
$$
g_n(\a_{[3]};\b)=\frac{(\psi-b_1+a_3)_n(\psi-b_2+a_3)_n}{n!}{_{3}F_2}\!\left(\!\left.\begin{array}{c}-n,b_3-a_1,b_3-a_2
\\\psi-b_1+a_3,\psi-b_2+a_3\end{array}\right|1\!\right).
$$
See \cite[formula~2.10]{Norlund} or \cite[p.\:48]{KLMeijer1}.
The coefficients $g_n(\a_{[1]};\b)$ and $g_n(\a_{[2]};\b)$ are obtained from the above by permutation of indices. Hence,
\begin{multline}\label{eq:G3033}
G^{3,0}_{3,3}\!\left(\!t~\vline\begin{array}{l}\b-1
\\\a-1\end{array}\!\!\right)=\frac{t^{a_3-1}(1-t)^{\psi-1}_{+}}{\Gamma(\psi)}\times
\\
\sum\limits_{n=0}^{\infty}\frac{(\psi-b_1+a_3)_n(\psi-b_2+a_3)_n}{(\psi)_nn!}{_{3}F_2}\!\left(\!\left.\begin{array}{c}-n,b_3-a_1,b_3-a_2
\\\psi-b_1+a_3,\psi-b_2+a_3\end{array}\right|1\!\right)(1-t)^n
\end{multline}
for $-\psi=a_1+a_2+a_3-b_1-b_2-b_3\notin\N_0$ and
\begin{multline*}
G^{3,0}_{3,3}\!\left(\!t~\vline\begin{array}{l}\b-1
\\\a-1\end{array}\!\!\right)=t^{a_3-1}\frac{(b_2+b_3-a_1-a_2)_{m+1}(b_1+b_3-a_1-a_2)_{m+1}}{(m+1)!}\times
\\
\sum\limits_{n=0}^{\infty}\frac{(1-b_1+a_3)_{n}(1-b_2+a_3)_{n}}{(m+2)_nn!}
{_{3}F_2}\left(\left.\!\!\begin{array}{c}-n-m-1,b_3-a_1,b_3-a_2
\\-m-b_1+a_3,-m-b_2+a_3\end{array}\right|1\!\right)(1-t)^n
\end{multline*}
for $-\psi=m\in\N_0$. Termwise integration and  application of (\ref{eq:fractint}) lead to
\begin{multline*}
G^{3,1}_{4,4}\!\left(\!t~\vline\begin{array}{l}n,\b+n-1
\\\a+n-1,0\end{array}\!\!\right)
\\
=
\frac{t^{a_3-1+n}}{(a_3)_n}\sum\limits_{k=0}^{\infty}\frac{(\psi-b_1+a_3)_k(\psi-b_2+a_3)_k}{\Gamma(\psi+k)k!}
{_{3}F_2}\left(\left.\!\!\begin{array}{c}-k,b_3-a_1,b_3-a_2
\\
\psi-b_1+a_3,\psi-b_2+a_3\end{array}\right|1\!\right)
{_2F_{1}}\!\left(\begin{array}{c}a_3,1-\psi-k
\\a_3+n\end{array}\!\!\vline\,t\right),
\end{multline*}
for $-\psi\notin\N_0$ and
\begin{multline*}
G^{3,1}_{4,4}\!\left(\!t~\vline\begin{array}{l}n,\b+n-1
\\\a+n-1,0\end{array}\!\!\right)
=\frac{(b_2+b_3-a_1-a_2)_{m+1}(b_1+b_3-a_1-a_2)_{m+1}t^{a_3-1+n}}{(m+1)!(a_3)_n}\times
\\
\sum\limits_{k=0}^{\infty}\frac{(1-b_1+a_3)_{k}(1-b_2+a_3)_{k}}{(m+2)_kk!}
{_{3}F_2}\left(\left.\!\!\begin{array}{c}-k-m-1,b_3-a_1,b_3-a_2
\\-m-b_1+a_3,-m-b_2+a_3\end{array}\right|1\!\right)
{_2F_{1}}\!\left(\begin{array}{c}a_3,-k
\\a_3+n\end{array}\!\!\vline\,t\right),
\end{multline*}
for $-\psi=m\in\N_0$.

Here, we again can employ (\ref{eq:GGF}) to express $G^{3,1}_{4,4}$ in terms of $G^{4,0}_{4,4}$ and ${}_4F_{3}$.
Next, we can write an explicit expansion for $G^{4,0}_{4,4}$ using (\ref{eq:Norlund}) and \cite[formula above (15)]{KLMeijer1} for the coefficients
$g_{j}(\a_{[k]};\b)$. Comparing the resulting expression with the one above leads to further transformation formulas for double sums
of hypergeometric functions. As these formulas are quite cumbersome we omit them here.

\section{A distribution generalizing $G_0$ by regularization at unity}

Define $\CB^{\infty}[0,1]$ to be the class of functions on $[0,1]$ that have derivatives of all orders which are all bounded on $[0,1]$.
If $\varphi\in\CB^{\infty}[0,1]$, then the integral (\ref{eq:Gphiintegral}) which we repeat for convenience here,
\makeatletter
\newcounter{tempcount}
\c@tempcount=\c@equation
\c@equation=0
\begin{equation}
\int_0^1\!\!G_0(t)\varphi(t)dt,~~\text{where}~~
G_0(t)=G^{p,0}_{p,p}\!\left(t\left\vert\begin{matrix}\b-1\\\a-1\end{matrix}\right.\right),
\end{equation}
\c@equation=\c@tempcount
\makeatother
converges (i.e. exists as an improper integral) if the next two conditions are satisfied:
\begin{equation*}\label{eq:a-psi-condition}
a=\min(\Re{a_1},\Re{a_2},\ldots,\Re{a_p})>0~~\text{and}~~\Re(\psi)=\Re\!\left[\sum\nolimits_{k=1}^{p}(b_k-a_k)\right]>0.
\end{equation*}
It also converges if the second condition is replaced by $\psi=0,-1,-2,\ldots$  These claims are immediate from the properties of $G_0(t)$
elaborated in \cite[section~2]{KLMeijer1}.  To convert the set $\CB^{\infty}[0,1]$ into a test function space, define the convergence in $\CB^{\infty}[0,1]$ as follows:
the sequence $\varphi_j$  converges to an element $\varphi\in\CB^{\infty}[0,1]$ if
$$
\max\limits_{x\in[0,1]}|\varphi^{(k)}_j(x)-\varphi^{(k)}(x)|\to0~\text{as}~j\to\infty
$$
for each nonnegative integer $k$. This space can be viewed as a space of restrictions of smooth periodic functions (say with period $2$) considered in \cite[Chapter~3, paragraph~2]{Beals} to the interval $[0,1]$. Then it follows from \cite[Theorem~2.1]{Beals} that this space is complete.  In this section we will define a regularization of the integral (\ref{eq:Gphiintegral}) assigning a finite value to it for arbitrary values of $a$ and $\psi$. Our regularization coincides with the analytic continuation in parameters $\a$ and $\b$, so that its application to the generalized Stieltjes, exponential or cosine Fourier kernel leads expectedly to the generalized hypergeometric functions.

\medskip

\noindent\textbf{Definition~1.} \emph{For arbitrary complex $\a$ and $\b$, $-\b\notin\N_0$,  choose a nonnegative integer $n>-\min(a,\Re(\psi))$, where $a$ and $\psi$
are given in \textup{(\ref{eq:a-psi-defined})}. Define a regularization of the integral \textup{(\ref{eq:Gphiintegral})} as the distribution $\G_1=\G_1(\a,\b)$
acting on a test function $\varphi\in\CB^{\infty}[0,1]$ according to the formula
\begin{equation}\label{eq:actionG}
\langle\G_1,\varphi\rangle=\sum_{k=0}^{n-1}{(-1)^k\varphi^{(k)}(1)\over k!}
{}_{p+1}F_p\left.\left(\begin{matrix}\a,-k \\\b\end{matrix}\right\vert 1\right)+{(-1)^{n}\Gamma(\b)\over\Gamma(\a)}\int_0^1\widehat{G}_n(t)\varphi^{(n)}(t)dt,
\end{equation}
where $\widehat{G}_n(t)$ is given in \textup{(\ref{eq:Gn-defined})}. If $n=0$ the finite sum in \textup{(\ref{eq:actionG})} is understood to be empty,
so that \textup{(\ref{eq:actionG})} reduces to  a multiple of \textup{(\ref{eq:Gphiintegral})}.
}

The asymptotic properties of $\widehat{G}_n(t)$ (at $t\to0$ and $t\to1$) contained in Propositions \ref{lm:Gn-asymp1} and  \ref{lm:Gn-asymp} show the correctness of
the above definition: the integral in (\ref{eq:actionG}) exists as a finite number for all $\varphi\in\CB^{\infty}[0,1]$ under the conditions stated in the definition.

When $n>0$ the above definition is motivated by the following argument.  Replace $\varphi(t)$ in (\ref{eq:Gphiintegral}) by its Taylor expansion at $t=1$:
$$
\varphi(t)=\sum_{k=0}^{n-1}\frac{\varphi^{(k)}(1)}{k!}(t-1)^k+\varphi_n(t),
$$
where $\varphi_n(t)$ is the Taylor remainder. Applying (\ref{eq:intG1-x}), we obtain the right hand side of (\ref{eq:actionG}), but with the second term replaced by
$$
{\Gamma(\b)\over\Gamma(\a)}\int_0^1G_0(t)\varphi_n(t)dt.
$$
Integrating by parts $n$ times and using (\ref{eq:hatGnG0}) and $\varphi_n^{(n)}(t)=\varphi^{(n)}(t)$, $\varphi^{(k)}_n(1)=\widehat{G}_{k+1}(0)=0$ for $k=0,1,\ldots,n-1$,
we obtain (\ref{eq:actionG}).  Therefore, for any $n\in\N$ the integral (\ref{eq:Gphiintegral}) equals the right hand side of (\ref{eq:actionG}) when $a>0$ and $\Re(\psi)>0$.
Moreover, the right hand side of (\ref{eq:actionG}) is an analytic function of the parameters $\a$ and meromorphic function of the parameters $\b$ with simple poles at $-b_i\in\N_0$;
hence, the right hand side of (\ref{eq:actionG}) gives an expression for the analytic continuation of (\ref{eq:Gphiintegral}) in $\a$ to the domain $\Re(\a)>-n$ and its
meromorphic continuation in $\b$ to the domain $\Re(\psi)>-n$. We conclude that the family of distributions $\G_1=\G_1(\a,\b)$ is analytic in the parameters $\a$
and meromorphic in $\b$ with simple poles at $-b_i\in\N_0$ in the above domain.

\textbf{Remark.}  The regularization defined in (\ref{eq:actionG}) can be easily seen to equal  the Hadamard finite part of the divergent integral (\ref{eq:Gphiintegral}).
See \cite{CostinFriedman,EstradaKanwal} for details.  However, we observe a new phenomenon here. In general, the Hadamard finite part constructed to overcome
divergence at unity does not alter the situation at other points, while formula (\ref{eq:actionG}) regularizes the integral (\ref{eq:Gphiintegral})
\emph{at both points, $0$ and $1$, simultaneously}.

\begin{theorem}\label{th:G1continous}
$\G_1$  is a continuous linear functional on $\CB^\infty[0,1]$ and its definition is independent of $n$.
\end{theorem}
\begin{proof} Linearity is obvious. For continuity, assume that $\varphi_j\to\varphi$ in $\CB^{\infty}[0,1]$ and define
$$
f_k:=\frac{1}{k!}{}_{p+1}F_p\left.\left(\begin{matrix}\a,-k \\\b\end{matrix}\right\vert 1\right).
$$
Then
\begin{multline*}
|\langle{\G_1,\varphi_j}\rangle-\langle{\G_1,\varphi}\rangle|=|\langle{\G_1,\varphi_j-\varphi}\rangle|\leq\sum\limits_{k=0}^{n-1}|f_k||\varphi^{(k)}_j(1)-\varphi^{(k)}(1)|
\\
+\max\limits_{x\in[0,1]}|\varphi^{(k)}_j(x)-\varphi^{(k)}(x)|\left\vert\frac{\Gamma(\b)}{\Gamma(\a)}\right\vert\int_0^1\vert\widehat{G}_n(t)\vert{dt}\to0~\text{as}~j\to \infty
\end{multline*}
by the definition of convergence in $\CB^{\infty}[0,1]$ and because the last integral in finite by Propositions \ref{lm:Gn-asymp1} and  \ref{lm:Gn-asymp}.
Finally, write $\G_{1,n}$ for the distribution $\G_1$ with $n$ terms in the sum (\ref{eq:actionG}) and $\G_{1,m}$ for $m\ne{n}$ terms. By definition we must choose $n,m>-\Re(\psi)$.
Assume without loss of generality that $n>m$ and let $\varphi$ be an arbitrary test function.  Integration by parts yields
\begin{multline*}
\langle{\G_{1,n},\varphi}\rangle-\langle{\G_{1,m},\varphi}\rangle=\sum\limits_{k=m}^{n-1}(-1)^kf_k\varphi^{(k)}(1)
+{(-1)^n\Gamma(\b)\over\Gamma(\a)}\int_0^1\widehat{G}_n(t)\varphi^{(n)}(t)dt
\\
-{(-1)^m\Gamma(\b)\over\Gamma(\a)}\int_0^1\widehat{G}_m(t)\varphi^{(m)}(t)dt=\sum\limits_{k=m}^{n-1}(-1)^kf_k\varphi^{(k)}(1)
\\
+{\Gamma(\b)\over\Gamma(\a)}\left\{(-1)^n\int_0^1\widehat{G}_n(t)\varphi^{(n)}(t)dt-(-1)^m\left.\widehat{G}_{m+1}(t)\varphi^{(m)}(t)\right\vert_0^1
+(-1)^m\int_0^1\widehat{G}_{m+1}(t)\varphi^{(m+1)}(t)dt\right\}
\\
-{(-1)^m\Gamma(\b)\over\Gamma(\a)}\int_0^1\widehat{G}_m(t)\varphi^{(m)}(t)dt=\sum\limits_{k=m+1}^{n-1}(-1)^kf_k\varphi^{(k)}(1)
\\
+{(-1)^n\Gamma(\b)\over\Gamma(\a)}\int_0^1\widehat{G}_n(t)\varphi^{(n)}(t)dt-{(-1)^{m+1}\Gamma(\b)\over\Gamma(\a)}\int_0^1\widehat{G}_{m+1}(t)\varphi^{(m+1)}(t)dt,
\end{multline*}
where we used $\widehat{G}_{m+1}(0)=0$, $\widehat{G}_{m+1}(1)=\Gamma(\a)f_m/\Gamma(\b)$ (by Proposition~\ref{lm:Gnmoments}).
Repeating integration by parts $(n-m)$ times and using the above calculation clearly leads to $\langle{\G_{1,n},\varphi}\rangle-\langle{\G_{1,m},\varphi}\rangle=0$.
\end{proof}

\begin{theorem}\label{th:G-hypergeom}
For arbitrary $\a\in\C^{p}$, $\b\in\C^{p}$, $-\b\notin\N_0$, and $\c\in\C^u$, $\d\in\C^s$,$-\d\notin\N_0$, choose a nonnegative
integer $n>-\min(a,\Re(\psi))$,
where $a$ and $\psi$ are defined in \textup{(\ref{eq:a-psi-defined})}. Then
\begin{equation}\label{eq:uFsaction}
\langle\G_1(\a,\b),{}_uF_{s}\!\left(\c;\d;-\!zt\right)\rangle
={}_{u+p}F_{s+p}\!\left.\left(\begin{matrix}\a,\c\\\b,\d\end{matrix}\right\vert -\!z\!\right)
\end{equation}
for all $z\in\C$ if $u\le{s}$ and for all $z\in\C\!\setminus\!(-\infty,-1]$ if $u=s+1$.
\end{theorem}
\begin{proof} Indeed, for $\Re(\psi)>0$ and $\Re(\a)>0$ the action of $\G_1$ reduces to the integral (\ref{eq:Gphiintegral}), so that formula (\ref{eq:uFsaction})
coincides with \cite[(4)]{KarpJMS}. For general $\a$ and $\psi$ the claim then follows by analytic continuation in
parameters, as both sides of (\ref{eq:uFsaction}) are analytic in $\a$ and meromorphic in $\b$.  Alternative direct proof can be furnished by using the
definition of $\G_1$ and applying formula (\ref{eq:Gnhypergeomint}) to the integral term and (\ref{eq:FFsummation}) to the resulting sum.
\end{proof}

The most important particular cases of Theorem~\ref{th:G-hypergeom} are given in the next corollary.
\begin{corollary}\label{cr:G1-representations}
Under conditions of Theorem~\textup{\ref{th:G-hypergeom}} we have\textup{:}
\begin{equation}\label{eq:G-p+1Fp}
\langle\G_1(\a,\b),(1+zt)^{-\sigma}\rangle={}_{p+1}F_p\left.\left(\begin{matrix}\sigma,\a\\\b\end{matrix}\right\vert -z\right)
\end{equation}
for all $z\in\C\!\setminus\!(-\infty,-1]$, $\sigma\in\C$,
\begin{equation}\label{eq:G-pFp}
\langle\G_1(\a,\b),e^{-zt}\rangle={}_{p}F_p\left.\left(\begin{matrix}\a\\\b\end{matrix}\right\vert -z\right)
\end{equation}
and
\begin{equation}\label{eq:G-p-1Fp}
\langle\G_1(\a,\b),\cos(2\sqrt{zt})\rangle={}_{p-1}F_p\left.\left(\begin{matrix}\hat{\a}\\\b\end{matrix}\right\vert -z\right)
\end{equation}
for all  $z\in\C$, where $\a=(\hat{\a},1/2)$ in the last formula.
\end{corollary}
\begin{proof} Indeed, (\ref{eq:G-p+1Fp}) and (\ref{eq:G-pFp}) are obviously particular cases of (\ref{eq:uFsaction}).  The last formula (\ref{eq:G-p-1Fp})
is also a particular case of (\ref{eq:uFsaction}) in view of $\a=(\hat{\a},1/2)$ and $\cos\left(2\sqrt{zt}\right)={}_0F_1(-;1/2;-zt)$.
\end{proof}

The combination of Corollary~\ref{cr:G1-representations} with Proposition~\ref{pr:signG} leads to the decomposition
formulas presented in the next corollary.

\begin{corollary}\label{cr:positivemeasure-repr}
Suppose that $\a$, $\b$ are arbitrary real vectors of size $p$ such that $-\b\notin\N_0$ and $\sigma$ is any real number.  Then there exists
$N\in\N_0$, such that for all $n\ge{N}$,
$$
\frac{1}{\Gamma(\b)}{}_{p+1}F_p\left.\left(\begin{matrix}\sigma,\a\\\b\end{matrix}\right\vert -z\right)
=\frac{1}{\Gamma(\b)}\sum_{k=0}^{n-1}{z^k(\sigma)_k\over(z+1)^{\sigma+k}k!}{}_{p+1}F_p\left(\left.\begin{matrix}\a,-k\\\b\end{matrix}\right\vert1\right)
+(-1)^{\eta}(\sigma)_nz^n\int_0^1\frac{\mu_n(dt)}{(1+zt)^{\sigma+n}},
$$
$$
\frac{1}{\Gamma(\b)}{}_{p}F_p\left.\left(\begin{matrix}\a\\\b\end{matrix}\right\vert -z\right)
=\frac{e^{-z}}{\Gamma(\b)}\sum_{k=0}^{n-1}{z^k\over k!}
{}_{p+1}F_p\left.\left(\begin{matrix}\a,-k \\\b\end{matrix}\right\vert 1\right)
+(-1)^{\eta}z^{n}\int_0^1e^{-zt}\mu_n(dt)
$$
and, with $\a=(\hat{\a},1/2)$,
\begin{multline*}
\frac{1}{\Gamma(\b)}{}_{p-1}F_p\left.\left(\begin{matrix}\hat{\a}\\\b\end{matrix}\right\vert -z\right)
=\frac{1}{\Gamma(\b)}\sum\limits_{k=0}^{n-1}\frac{z^k{}_0F_1(-;k+1/2;-z)}{(1/2)_kk!}
{}_{p+1}F_p\left.\left(\begin{matrix}-k,\a\\\b\end{matrix}\right\vert1\right)
\\
+(-1)^{\eta}\frac{z^n}{(1/2)_n}\int_0^1{}_0F_1(-;n+1/2;-zt)\mu_n(dt),
\end{multline*}
where
$$
\mu_n(dt):={(-1)^{\eta}\over\Gamma(\a)}G^{p,1}_{p+1,p+1}\left(t\left|\begin{array}{l}\!\!n,\b+n-1\!\!\\\!\!\a+n-1,0\!\!\end{array}\right.\right)dt
$$
is a positive measure. The number $\eta$ is defined in \textup{(\ref{eq:et})}.
\end{corollary}

\section{Alternative regularizations tailored for the Bessel type functions}

The finite sum in the decomposition corresponding to $\langle\G_1(\a,\b),\cos(\sqrt{zt})\rangle$ given by the last formula
of Corollary~\ref{cr:positivemeasure-repr} contains a non-elementary Bessel function ${}_0F_1$.  It seems desirable, however,
to derive an alternative representation containing only the elementary cosine function, in particular, when studying zeros of the Bessel type hypergeometric
function.  To this end, start with \cite[(5)]{KLMeijer1}:
\begin{equation*}
{}_{p-1}F_p\left.\left(\begin{matrix}\hat{\a}\\\b\end{matrix}\:\right\vert -z\right)
={\Gamma(\b)\over\sqrt{\pi}\Gamma(\hat{\a})}\int_0^1\cos(2\sqrt{zt})G^{p,0}_{p,p}\left(t\left|\begin{array}{l}\!\!\b-1\!\!\\\!\!\hat{\a}-1,-1/2\!\!\end{array}\right.\right)\!dt.
\end{equation*}
Setting $t=u^2$ and changing $z\to{z^2/4}$ we get
\begin{equation}\label{eq:p-1Fpcos}
{}_{p-1}F_p\left.\left(\begin{matrix}\hat{\a}\\\b\end{matrix}\:\right\vert -z^2/4\right)
={2\Gamma(\b)\over\sqrt{\pi}\Gamma(\hat{\a})}\int_0^1\cos(zu)
G^{p,0}_{p,p}\left(u^2\left|\begin{array}{l}\!\!\b-1/2\!\!\\\!\!\hat{\a}-1/2,0\!\!\end{array}\right.\right)\!du.
\end{equation}
Hence, we need to regularize integrals of the form
\begin{equation}\label{eq:Gsquare-int}
\int_0^1G^{p,0}_{p,p}\left(u^2\left|\begin{array}{l}\!\!\b-1/2\!\!\\\!\!\a-1/2\!\!\end{array}\right.\right)\phi(u)du.
\end{equation}
Assuming $\phi\in\CB^{\infty}[0,1]$, we apply Taylor's theorem with integral remainder to $\phi$ in the neighborhood of  $u=1$ to obtain
\begin{multline*}
\int_0^1G^{p,0}_{p,p}\left(u^2\left|\begin{array}{l}\!\!\b-1/2\!\!\\\!\!\a-1/2\!\!\end{array}\right.\right)\phi(u)du
\\
=\int_0^1G^{p,0}_{p,p}\left(u^2\left|\begin{array}{l}\!\!\b-1/2\!\!\\\!\!\a-1/2\!\!\end{array}\right.\right)\left[\sum\limits_{k=0}^{n-1}\frac{\phi^{(k)}(1)}{k!}(u-1)^k
+\frac{1}{(n-1)!}\int\limits_{1}^{u}(u-t)^{n-1}\phi^{(n)}(t)dt\right]\!du
\\
=\!\sum\limits_{k=0}^{n-1}\frac{\phi^{(k)}(1)}{(-1)^kk!}\!\int\limits_0^1\!\!G^{p,0}_{p,p}\!\left(\!u^2\!\left|\begin{array}{l}\!\!\b-1/2\!\!\\\!\!\a-1/2\!\!\end{array}\right.\right)(1-u)^k\!du
+\frac{(-1)^n}{\Gamma(n)}\!\int\limits_0^1\!\!G^{p,0}_{p,p}\!\left(\!u^2\!\left|\begin{array}{l}\!\!\b-1/2\!\!\\\!\!\a-1/2\!\!\end{array}\right.\right)\!du
\!\!\int\limits_{u}^{1}\!(t-u)^{n-1}\phi^{(n)}(t)dt.
\end{multline*}
By substitution $t=u^2$ and separation of odd and even terms, the leftmost integral in the last expression is elaborated as follows:
\begin{multline*}
\int_0^1G^{p,0}_{p,p}\left(u^2\left|\begin{array}{l}\!\!\b-1/2\!\!\\\!\!\a-1/2\!\!\end{array}\right.\right)(1-u)^k\!du=\sum\limits_{j=0}^{k}(-1)^j\binom{k}{j}\int_0^1G^{p,0}_{p,p}\left(u^2\left|\begin{array}{l}\!\!\b-1/2\!\!\\\!\!\a-1/2\!\!\end{array}\right.\right)u^j\!du
\\
=\frac{1}{2}\sum\limits_{j=0}^{k}(-1)^j\binom{k}{j}\int_0^1G^{p,0}_{p,p}\left(t\left|\begin{array}{l}\!\!\b-1/2\!\!\\\!\!\a-1/2\!\!\end{array}\right.\right)t^{j/2-1/2}\!dt=\sum\limits_{j=0}^{k}
\frac{(-k)_j\Gamma(\a+j/2)}{2\Gamma(\b+j/2)j!}
\\
=\sum\limits_{\stackrel{j=0}{j=2m}}^{k}
\frac{(-k)_j\Gamma(\a+j/2)}{2\Gamma(\b+j/2)j!}
+
\sum\limits_{\stackrel{j=0}{j=2m+1}}^{k}
\frac{(-k)_j\Gamma(\a+j/2)}{2\Gamma(\b+j/2)j!}
\\
=\!\!\sum\limits_{m=0}^{k/2}\!\frac{(-k/2)_m(-k/2+1/2)_m\Gamma(\a+m)}{2(1/2)_m\Gamma(\b+m)m!}
+\!\!\!\!
\sum\limits_{m=0}^{(k-1)/2}\!\frac{(-k)(-k/2+1)_m(-k/2+1/2)_m\Gamma(\a+1/2+m)}{2\Gamma(\b+1/2+m)(3/2)_mm!}
\\
=\!\frac{\Gamma(\a)}{2\Gamma(\b)}{}_{p+2}F_{p+1}\!\left(\begin{matrix}-k/2,-k/2+1/2,\a\\1/2,\b\end{matrix}\right)
-\frac{k\Gamma(\a+1/2)}{2\Gamma(\b+1/2)}
{}_{p+2}F_{p+1}\!\left(\begin{matrix}-k/2+1,-k/2+1/2,\a+1/2\\ 3/2,\b+1/2\end{matrix}\right)\!,
\end{multline*}
where we utilized the shorthand notation ${}_{p}F_{q}(\a;\b)={}_{p}F_{q}(\a;\b;1)$ and the easily verifiable identities
$$
(2m)!=4^m(1/2)_mm!,~~~(2m+1)!=4^m(3/2)_mm!,
$$
$$
(-k)_{2m}=4^m(-k/2)_m(-k/2+1/2)_m,~~~(-k)_{2m+1}=(-k)4^m(-k/2+1)_m(-k/2+1/2)_m.
$$
Further, using \cite[formula~2.24.2.2]{PBM3} for the second term, we get
\begin{multline*}
\int\limits_0^1\!\!G^{p,0}_{p,p}\left(u^2\left|\begin{array}{l}\!\!\b-1/2\!\!\\\!\!\a-1/2\!\!\end{array}\right.\right)\!du
\!\int\limits_{u}^{1}(t-u)^{n-1}\phi^{(n)}(t)dt
=\!\!\int\limits_0^1\!\!\phi^{(n)}(t)dt\int\limits_{0}^{t}\!G^{p,0}_{p,p}\left(u^2\left|\begin{array}{l}\!\!\b-1/2\!\!\\\!\!\a-1/2\!\!\end{array}\right.\right)(t-u)^{n-1}\!du
\\
=\frac{(n-1)!}{2^n}\int_0^1
G^{p,2}_{p+2,p+2}\left(t^2\left|\begin{array}{l}\!\!n/2,(n+1)/2,\b+(n-1)/2\!\!\\\!\!\a+(n-1)/2,0,1/2\!\!\end{array}\right.\right)\phi^{(n)}(t)dt.
\end{multline*}
Combining these formulas we can define the regularization of the integral (\ref{eq:Gsquare-int}) as follows.

\medskip

\noindent\textbf{Definition~2.} \emph{For arbitrary complex $\a$ and $\b$, $-\b\notin\N_0$,  choose a nonnegative integer $n>-\min(a,\Re(\psi))$,
where $a$ and $\psi$ are given in \textup{(\ref{eq:a-psi-defined})}. Define a regularization of the integral
\textup{(\ref{eq:Gsquare-int})} as the distribution $\G_b^{1}=\G_b^{1}(\a,\b)$
acting on a test function $\phi\in\CB^{\infty}[0,1]$ according to the formula
\begin{multline}\label{eq:Gb1}
\langle\G_b^{1},\phi\rangle=\sum\limits_{k=0}^{n-1}\frac{\phi^{(k)}(1)}{(-1)^kk!}\biggl\{\frac{\Gamma(\a)}{2\Gamma(\b)}
{}_{p+2}F_{p+1}\!\left(\begin{matrix}-k/2,-k/2+1/2,\a\\1/2,\b\end{matrix}\right)
\\
-\frac{k\Gamma(\a+1/2)}{2\Gamma(\b+1/2)}
{}_{p+2}F_{p+1}\!\left(\begin{matrix}-k/2+1,-k/2+1/2,\a+1/2\\ 3/2,\b+1/2\end{matrix}\right)\biggr\}
\\
+\frac{(-1)^n}{2^n}\int_0^1
G^{p,2}_{p+2,p+2}\left(t^2\left|\begin{array}{l}\!\!n/2,(n+1)/2,\b+(n-1)/2\!\!\\\!\!\a+(n-1)/2,0,1/2\!\!\end{array}\right.\right)\phi^{(n)}(t)dt.
\end{multline}
If $n=0$ the finite sum in \textup{(\ref{eq:Gb1})} is understood to be empty,
so that \textup{(\ref{eq:Gb1})} reduces to  a multiple of \textup{(\ref{eq:Gsquare-int})}.
}

\smallskip

An argument similar to that given in the previous section shows that $\G_b^{1}$ is a continuous linear functional
on $\CB^{\infty}[0,1]$, whose definition is independent of $n$.  Furthermore, $\langle\G_b^{1},\phi\rangle$ coincides
with the analytic continuation of (\ref{eq:Gsquare-int}) in the parameters $\a$ and $\b$.

Using
$$
\frac{\partial^k}{\partial{u}^k}\cos(zu)=z^k\cos(zu+\pi{k/2})
$$
and setting $\a=(\hat{\a},1/2)$, we then obtain
\begin{multline}\label{eq:p-1Fpsecond}
{}_{p-1}F_p\left.\left(\begin{matrix}\hat{\a}\\\b\end{matrix}\:\right\vert -z^2/4\right)
=\sum\limits_{k=0}^{n-1}\frac{z^k\cos(z+\pi{k/2})}{(-1)^kk!}\biggl\{
{}_{p+2}F_{p+1}\!\left(\begin{matrix}-k/2,-k/2+1/2,\a\\1/2,\b\end{matrix}\right)
\\
-k\frac{\Gamma(\b)\Gamma(\a+1/2)}{\Gamma(\a)\Gamma(\b+1/2)}
{}_{p+2}F_{p+1}\!\left(\begin{matrix}-k/2+1,-k/2+1/2,\a+1/2\\ 3/2,\b+1/2\end{matrix}\right)\biggr\}
\\
+\frac{(-z)^n\Gamma(\b)}{2^{n-1}\Gamma(\a)}\int_0^1
G^{p,2}_{p+2,p+2}\left(t^2\left|\begin{array}{l}\!\!n/2,(n+1)/2,\b+(n-1)/2\!\!\\\!\!\a+(n-1)/2,0,1/2\!\!\end{array}\right.\right)\cos(zt+\pi{n/2})dt.
\end{multline}

We will use a particular case of (\ref{eq:p-1Fpsecond}) to extract some information about the zeros of the function on the left hand side.  In order to do his,
we need to recall some facts regarding the positivity of the Meijer-N{\o}rlund function $G^{p,0}_{p,p}$. We follow \cite[Property~9]{KLMeijer1}.
The inequality
$$
G^{p,0}_{p,p}\!\left(\!x~\vline\begin{array}{l}\b\\\a\end{array}\!\!\right)\ge0~~\text{for}~~0<x<1
$$
holds if $v_{\a,\b}(t)=\sum_{j=1}^{p}(t^{a_j}-t^{b_j})\ge0$ for $t\in[0,1]$. See \cite[Theorem~2]{KarpJMS} for a proof of this fact and \cite[Section~2]{KPCMFT} for further details.
Note also that $v_{\a,\b}(t)\ge0$ implies that $\psi=\sum_{j=1}^{p}(b_j-a_j)\ge0$.  For given $\a$, $\b$ the inequality $v_{\a,\b}(t)\ge0$ is not easy to verify other than numerically.
However, several sufficient conditions for $v_{\a,\b}(t)\ge0$ expressed directly in terms of $\a$, $\b$ are known.
In particular, according to  \cite[Theorem~10]{Alzer} $v_{\a,\b}(t)\ge0$ on $[0,1]$ if
\begin{equation}\label{eq:amajorb}
\begin{split}
& 0<a_1\leq{a_2}\leq\cdots\leq{a_p},~~
0<b_1\leq{b_2}\leq\cdots\leq{b_p},
\\
&\sum\limits_{i=1}^{k}a_i\leq\sum\limits_{i=1}^{k}b_i~~\text{for}~~k=1,2\ldots,p.
\end{split}
\end{equation}
These inequalities are known as the weak supermajorization and are abbreviated as $\b\!\prec^W\!\a$.   Further sufficient conditions can be found in
\cite[Property~9]{KLMeijer1} and \cite[Section~2]{KPCMFT}.

\begin{lemma} \label{lm:Gincrease}
Suppose $\alpha\ge0$, $\beta-\alpha\ge1$, $\a>\beta-1$ and $v_{\a,\b}(t)\ge0$ on $[0,1]$
\emph{(}in particular, $\b-\beta+1\prec^W\a-\beta+1$ is sufficient\emph{)}. Then the function
$$
x\to G^{p,1}_{p+1,p+1}\left(x\left|\begin{array}{l}\!\!\beta,\b\\\!\!\a,\alpha\end{array}\!\!\right.\right)
$$
is positive and increasing on $(0,1)$.
\end{lemma}
\begin{proof} Set $\gamma=\beta-\alpha\ge1$, $\b'=\b-\beta$, $\a'=\a-\beta$. We have
$$
G^{p,1}_{p+1,p+1}\left(x\left|\begin{array}{l}\!\!\beta,\b\\\!\!\a,\alpha\end{array}\!\!\right.\right)
=x^{\alpha}G^{p,1}_{p+1,p+1}\left(x\left|\begin{array}{l}\!\!\gamma,\b'+\gamma\\\!\!\a'+\gamma,0\end{array}\!\!\right.\right)
=\frac{x^{\alpha}}{\Gamma(\gamma)}\int_0^x(x-t)^{\gamma-1}G^{p,0}_{p,p}\left(t\left|\begin{array}{l}\!\!\b'\\\!\!\a'\end{array}\!\!\right.\right)dt
$$
according to \cite[2.24.2.2]{PBM3}.  The nonnegativity of the $G$ function in the integrand combined with $\gamma-1\ge0$ completes the proof.
\end{proof}

\begin{theorem}\label{th:p-1Fpestimate}
Let $\hat{\a}$, $\b$ be positive vectors. Set $\a=(\hat{\a},1/2)$ and assume that $v_{\a,\b}(t)\ge0$ on $[0,1]$ \emph{(}in particular, $\b\prec^W\a$ is sufficient\emph{)}.
Then all zeros of
$$
f(z)={}_{p-1}F_p\left.\left(\begin{matrix}\hat{\a}\\\b\end{matrix}\:\right\vert -z^2/4\right)-\cos(z)
$$
are real and simple.  Each interval $(\pi,2\pi)$, $(2\pi,3\pi)$ contains exactly one zero of $f$ and there are no other zeros  except the trivial one at $z=0$.
\end{theorem}
\begin{proof}  Using (\ref{eq:p-1Fpsecond}) for $n=1$ we obtain
$$
{}_{p-1}F_p\left.\left(\begin{matrix}\hat{\a}\\\b\end{matrix}\:\right\vert -z^2/4\right)
=\cos(z)
+\frac{z\Gamma(\b)}{\Gamma(\a)}\int_0^1
G^{p,1}_{p+1,p+1}\left(t^2\left|\begin{array}{l}\!\!1,\b\!\!\\\!\!\a,0\!\!\end{array}\right.\right)\sin(zt)dt,
$$
where $\a=(\hat{\a},1/2)$.  According to Lemma~\ref{lm:Gincrease} the $G$ function in the integrand is positive
and increasing and is  obviously not a step function. The claim now follows by \cite[Theorem~2.2.2]{Sedlet2005}.
\end{proof}

In a nice recent paper \cite{ChoYun} the authors found the exact range of positive parameters $\alpha$, $\beta_1$, $\beta_2$ that ensure the inequality
${}_1F_{2}(\alpha;\beta_1,\beta_2;x)\ge0$ for all real $x$.  This range can be described as follows: for $\alpha>0$ let $P_{\alpha}$ denote  the convex hull
of the points $(\alpha_m,\infty)$, $(\alpha_m,\alpha_M)$, $(\alpha_M,\alpha_m)$, $(\infty,\alpha_m)$ in the plane $(\beta_1,\beta_2)$,
where $\alpha_m=\min(2\alpha,\alpha+1/2)$,  $\alpha_M=\max(2\alpha,\alpha+1/2)$.
Then ${}_1F_{2}(\alpha;\beta_1,\beta_2;x)\ge0$ for $\alpha,\beta_1,\beta_2>0$  iff $(\beta_1,\beta_2)\in{P_{\alpha}}$.  In the final section the authors extend their results
to the Bessel type generalized hypergeometric function ${}_{p-1}F_{p}$.  Their extension theorem can be strengthen as follows.

\begin{theorem}\label{th:ChoYun}
Suppose that $\alpha>0$,  $(\beta_1,\beta_2)\in{P_{\alpha}}$, $\a,\b\in\R^{p-1}$ with $\a>0$ and $v_{\a,\b}(t)\ge0$ on $[0,1]$
\emph{(}in particular, $\b\prec^W\a$ is sufficient\emph{)}.
Then
$$
{}_{p}F_{p+1}\left.\left(\begin{matrix}\alpha,\a\\\beta_1,\beta_2,\b\end{matrix}\:\right\vert x\right)\ge0
$$
for all $x\in\R$.
\end{theorem}
\begin{proof}
Indeed, according to \cite[Theorem~1]{KarpJMS}
$$
{}_{p}F_{p+1}\left.\left(\begin{matrix}\alpha,\a\\\beta_1,\beta_2,\b\end{matrix}\:\right\vert x\right)
=\frac{\Gamma(\b)}{\Gamma(\a)}\int_0^1\!\!\!{}_{1}F_{2}\left.\left(\begin{matrix}\alpha\\\beta_1,\beta_2\end{matrix}\:\right\vert xt\right)
G^{p-1,0}_{p-1,p-1}\left(t\left|\begin{array}{l}\!\b\!\\\!\a\!\end{array}\right.\right)\frac{dt}{t}
$$
for $\a>0$ and $\psi=\sum_k(b_k-a_k)>0$,  and
$$
{}_{p}F_{p+1}\left.\left(\begin{matrix}\alpha,\a\\\beta_1,\beta_2,\b\end{matrix}\:\right\vert x\right)
=\frac{\Gamma(\b)}{\Gamma(\a)}\left[
{}_{1}F_{2}\left.\left(\begin{matrix}\alpha\\\beta_1,\beta_2\end{matrix}\:\right\vert x\right)
+\int_0^1\!\!\!{}_{1}F_{2}\left.\left(\begin{matrix}\alpha\\\beta_1,\beta_2\end{matrix}\:\right\vert xt\right)
G^{p-1,0}_{p-1,p-1}\left(t\left|\begin{array}{l}\!\b\!\\\!\a\!\end{array}\right.\right)\frac{dt}{t}\right]
$$
for $\a>0$ and $\psi=0$.  Hence, the claim follows \cite[Theorem~6.1]{ChoYun} in view of nonnegativity of the $G$ function
in the integrand valid according to \cite[Theorem~2]{KarpJMS} or \cite[Property~9]{KLMeijer1}.
\end{proof}

\bigskip

\paragraph{Acknowledgements.}
The research of the first author has been supported by the Russian Science Foundation under the project 14-11-00022.
The research of the second author has been supported by the Spanish \emph{Ministry
of "Econom\'{\i}a y Competitividad"} under the project MTM2014-53178  and by the {\it Universidad P\'ublica de Navarra}.

\end{document}